\documentclass{gOMS2e}

\usepackage{epstopdf}
\usepackage{subfigure}
\usepackage[framemethod=tikz]{mdframed}

\def\dom{{\rm dom \,}}
\def\beq{\begin{equation}}
\def\eeq{\end{equation}}

\def\E{\mathbb{E}}

\def\R{\mathbb{R}}

\newtheorem{assumption}{Assumption}
\newtheorem{example}{Example}
\def\ba{\begin{array}}
\def\ea{\end{array}}
\def\beann{\begin{eqnarray*}}
\def\eeann{\end{eqnarray*}}
\def\bea{\begin{eqnarray}}
\def\eea{\end{eqnarray}}

\def\BT{\begin{theorem}}
\def\ET{\end{theorem}}
\def\BL{\begin{lemma}}
\def\EL{\end{lemma}}
\def\BC{\begin{corollary}}
\def\EC{\end{corollary}}
\def\BE{\begin{example}}
\def\EE{\end{example}}
\def\BD{\begin{definition}}
\def\ED{\end{definition}}
\def\BR{\begin{remark}}
\def\ER{\end{remark}}
\def\BAS{\begin{assumption}}
\def\EAS{\end{assumption}}
\def\BI{\begin{itemize}}
\def\EI{\end{itemize}}

\def\BMP{\begin{minipage}{9.5cm}}
\def\EMP{\end{minipage}}
\def\MPT{\begin{minipage}{11.5cm}}
\def\EPT{\end{minipage}}

\def\la{\langle}
\def\ra{\rangle}

\theoremstyle{plain}
\newtheorem{theorem}{Theorem}[section]
\newtheorem{corollary}[theorem]{Corollary}
\newtheorem{lemma}[theorem]{Lemma}

\theoremstyle{definition}
\newtheorem{definition}{Definition}

\theoremstyle{remark}
\newtheorem{remark}{Remark}

\begin{document}



\title{Tensor Methods for Finding Approximate Stationary Points of Convex Functions}

\author{
\name{G.N. Grapiglia\textsuperscript{a}$^{\ast}$\thanks{$^\ast$Corresponding author. Email: grapiglia@ufpr.br}, Yu. Nesterov\textsuperscript{b}}
\affil{\textsuperscript{a}Departamento de Matem\'atica, Universidade Federal do Paran\'a, Centro Polit\'ecnico, Cx. postal 19.081, 81531-980, Curitiba, Paran\'a, Brazil;\\
\textsuperscript{b}Center for Operations Research and Econometrics (CORE), Catholic University of Louvain
(UCL), 34 voie du Roman Pays, 1348 Louvain-la-Neuve, Belgium}
\received{August 18, 2020}
}

\maketitle

\begin{abstract}
In this paper we consider the problem of finding $\epsilon$-approximate stationary points of convex functions that are $p$-times differentiable with $\nu$-H\"{o}lder continuous $p$th derivatives. We present tensor methods with and without acceleration. Specifically, we show that the non-accelerated schemes take at most $\mathcal{O}\left(\epsilon^{-1/(p+\nu-1)}\right)$ iterations to reduce the norm of the gradient of the objective below a given $\epsilon\in (0,1)$. For accelerated tensor schemes we establish improved complexity bounds of $\mathcal{O}\left(\epsilon^{-(p+\nu)/[(p+\nu-1)(p+\nu+1)]}\right)$ and $\mathcal{O}\left(|\log(\epsilon)|\epsilon^{-1/(p+\nu)}\right)$, when the H\"{o}lder parameter $\nu\in [0,1]$ is known. For the case in which $\nu$ is unknown, we obtain a bound of  $\mathcal{O}\left(\epsilon^{-(p+1)/[(p+\nu-1)(p+2)]}\right)$ for a universal accelerated scheme. Finally, we also obtain a lower complexity bound of $\mathcal{O}\left(\epsilon^{-2/[3(p+\nu)-2]}\right)$ for finding $\epsilon$-approximate stationary points using $p$-order tensor methods.
\end{abstract}

\begin{keywords}
unconstrained minimization; high-order
methods, tensor methods; H\"older condition; worst-case complexity
\end{keywords}

\begin{classcode}49M15; 49M37; 58C15; 90C25; 90C30
\end{classcode}

\section{Introduction}
\setcounter{equation}{0}

\subsection{Motivation}

In this paper we study tensor methods for unconstrained optimization, i.e., methods in which the iterates are obtained by the (approximate) minimization of models defined from high-order Taylor approximations of the objective function. This type of methods is not new in the Optimization literature (see, e.g., \cite{SCH,BOU,BAES}). Recently, the interest for tensor methods has been renewed by the work in \cite{Birgin}, where $p$-order tensor methods were proposed for unconstrained minimization of \textit{nonconvex functions} with Lipschitz continuous $p$th derivatives. It was shown that these methods take at most $\mathcal{O}\left(\epsilon^{-\frac{p+1}{p}}\right)$ iterations to find an $\epsilon$-approximate first order stationary point of the objective, generalizing the bound of $\mathcal{O}\left(\epsilon^{-3/2}\right)$, originally established in \cite{NP} for the Cubic Regularization of Newton's Method ($p=2$). After \cite{Birgin}, several high-order methods have been proposed and analyzed for nonconvex optimization (see, e.g., \cite{CGT2,CHEN1,CHEN2,Mart}), resulting even in worst-case complexity bounds for the number of iterations that $p$-order methods need to generate approximate $q$th order stationary points \cite{CGT1,CGT3}.

More recently, in \cite{NES6}, $p$-order tensor methods with and without acceleration were proposed for unconstrained minimization of \textit{convex functions} with Lipschitz continuous $p$th derivatives. As it is usual in Convex Optimization, these methods aim the generation of a point $\bar{x}$ such that $f(\bar{x})-f^{*}\leq\epsilon$, where $f$ is the objective function, $f^{*}$ is its optimal value and $\epsilon>0$ is a given precision. Specifically, it was shown that the non-accelerated scheme takes at most $\mathcal{O}(\epsilon^{-1/p})$ iterations to reduce the functional residual below a given $\epsilon>0$, while the accelerated scheme takes at most $\mathcal{O}(\epsilon^{-1/(p+1)})$ iterations to accomplish the same task. Auxiliary problems in these methods consist in the minimization of a $(p+1)$-regularization of the $p$th order Taylor approximation of the objective, which is a multivariate polynomial. A remarkable result shown in \cite{NES6} (which distinguish this work from \cite{BAES}) is that, in the convex case, the auxiliary problems in tensor methods become convex when the corresponding regularization parameter is sufficiently large. Since \cite{NES6}, several high-order methods have been proposed for convex optimization (see, e.g., \cite{DOI2,DOI3,GN3,JIA}), including near-optimal methods \cite{BUB,GAS,JIA1,NES10,NES11} motivated by the second-order method in \cite{MS}. In particular, in \cite{GN3}, we have adapted and generalized the methods in \cite{GN,GN2,NES6} to handle convex functions with $\nu$-H\"{o}lder continuous $p$th derivatives. It was shown that the non-accelerated schemes take at most $\mathcal{O}(\epsilon^{-1/(p+\nu-1)})$ iterations to generate a point with functional residual smaller than a given $\epsilon\in (0,1)$, while the accelerated variants take only $\mathcal{O}(\epsilon^{-1/(p+\nu)})$ iterations when the parameter $\nu$ is explicitly used in the scheme. For the case in which $\nu$ is not known, we also proposed a universal accelerated scheme for which we established an iteration complexity bound of $\mathcal{O}(\epsilon^{-p/[(p+1)(p+\nu-1)]})$.

As a natural development, in this paper we present variants of the $p$-order methods ($p\geq 2$) proposed in \cite{GN3} that aim the generation of a point $\bar{x}$ such that $\|\nabla f(\bar{x})\|_{*}\leq\epsilon$, for a given threshold $\epsilon\in (0,1)$. In the context of nonconvex optimization, finding approximate stationary points is usually the best one can expect from local optimization methods. In the context of convex optimization, one motivation to search for approximate stationary points is the fact that the norm of the gradient may serve as a measure of feasibility and optimality when one applies the dual approach for solving constrained convex problems (see, e.g., \cite{NES8}). Another motivation comes from the inexact high-order proximal-point methods, recently proposed in \cite{NES10,NES11}, in which the iterates are computed as approximate stationary points of uniformly convex models.

Specifically, our contributions are the following:
\begin{enumerate}
\item We show that the non-accelerated schemes in \cite{GN3} take at most $\mathcal{O}\left(\epsilon^{-1/(p+\nu-1)}\right)$ iterations to reduce the norm of the gradient of the objective below a given $\epsilon\in (0,1)$, when the objective is convex, and $\mathcal{O}\left(\epsilon^{-(p+\nu)/(p+\nu-1)}\right)$ iterations, when $f$ is nonconvex. These complexity bounds extend our previous results reported in \cite{GN} for regularized Newton methods (case $p=2$). Moreover, our complexity bound for the nonconvex case agrees in order with the bounds obtained in \cite{Mart} and \cite{CGT2} for different tensor methods.
\item[]
\item For accelerated tensor schemes we establish improved complexity bounds of $\mathcal{O}\left(\epsilon^{-(p+\nu)/[(p+\nu-1)(p+\nu+1)]}\right)$, when the H\"{o}lder parameter $\nu\in [0,1]$ is known. This result generalizes the bound of $\mathcal{O}\left(\epsilon^{-2/3}\right)$ obtained in \cite{NES8} for the accelerated gradient method ($p=\nu=1$). In contrast, when $\nu$ is unknown, we prove a bound of  $\mathcal{O}\left(\epsilon^{-(p+1)/[(p+\nu-1)(p+2)]}\right)$ for a universal accelerated scheme.
\item[]
\item For the case in which $\nu$ and the corresponding H\"{o}lder constant are known, we propose tensor schemes for the composite minimization problem. In particular, we prove a bound of $\mathcal{O}\left(R^{\frac{p+\nu-1}{p+\nu}}|\log_{2}(\epsilon)|\epsilon^{-\frac{1}{p+\nu}}\right)$ iterations, where $R$ is an upper bound for the initial distance to the optimal set. This result generalizes the bounds obtained in \cite{NES8} for first-order and second-order accelerated schemes combined with a regularization approach ($p=1,2$ and $\nu=1$). We also prove a bound of $\mathcal{O}\left(S^{\frac{p+\nu-1}{p+\nu}}|\log_{2}(\epsilon)|\epsilon^{-1}\right)$ iterations, where $S$ is an upper bound for the initial functional residual.
\item[]
\item Considering the same class of difficult functions described in \cite{GN3}, we derive a lower complexity bound of $\mathcal{O}\left(\epsilon^{-2/[3(p+\nu)-2]}\right)$ iterations (in terms of the initial distance to the optimal set), and a lower complexity bound of $\mathcal{O}\left(\epsilon^{-2(p+\nu)/[3(p+\nu)-2]}\right)$ iterations (in terms of the initial functional residual), for $p$-order tensor methods to find $\epsilon$-approximate stationary points of convex functions with $\nu$-H\"{o}lder continuous $p$th derivatives. These bounds generalize the corresponding bounds given in \cite{CARMON} for first-order methods ($p=\nu=1$).
\end{enumerate}

The paper is organized as follows. In section 2, we define our problem. In section 3, we present complexity results for tensor schemes without acceleration. In section 4, we present complexity results for accelerated schemes. In section 5 we analyze tensor schemes for the composite minimization problem. Finally, in section 6, we establish lower complexity bounds for tensor methods find $\epsilon$-approximate stationary points of convex functions under the H\"{o}lder condition. Some auxiliary results are left in the Appendix.

\subsection{Notations and Generalities}

Let $\E$ be a finite-dimensional
real vector space, and $\E^*$ be its {\em dual} space. 
We denote by $\la s, x \ra$ the value of the linear functional 
$s \in \E^*$ at point $x \in \E$. Spaces $\E$ and $\E^{*}$ are
equipped with conjugate Euclidean norms:
\begin{equation}
\| x \| = \la B x, x \ra^{1/2},\; x\in \E, \quad \| s
\|_* \; = \; \la s, B^{-1} s \ra^{1/2}, \; s \in \E^*.
\label{eq:0.1}
\end{equation}
where  $B:\E \to \E^*$ is a self-adjoint positive definite operator ($B \succ 0$).
For a smooth function $f:\E\to \R$, denote by $\nabla f(x)$
its gradient, and by $\nabla^{2}f(x)$ its Hessian evaluated at 
point $x\in\E$. Then $\nabla f(x)\in\E^{*}$ and 
$\nabla^{2}f(x)h\in\E^{*}$ for $x,h\in\E$.

For any integer $p\geq 1$, denote by
\begin{equation*}
D^{p}f(x)[h_{1},\ldots,h_{p}]
\end{equation*} 
the directional derivative of function $f$ at $x$ along directions
$h_{i}\in\E$, $i=1,\ldots,p$. For any $x\in\text{dom}\,f$ and $h_{1},h_{2}\in\E$ we have
\begin{equation*}
Df(x)[h_{1}]=\langle\nabla f(x),h_{1}\rangle\quad\text{and}\quad D^{2}f(x)[h_{1},h_{2}]=\langle\nabla^{2}f(x)h_{1},h_{2}\rangle.
\end{equation*}
If $h_{1}=\ldots=h_{p}=h\in\E$, we denote $D^{p}f(x)[h_{1},\ldots,h_{p}]$ by $D^{p}f(x)[h]^{p}$. Using this notation, the $p$th order Taylor approximation of function $f$ at $x\in\E$ can be written as follows:
\begin{equation}
f(x+h)=\Phi_{x,p}(x+h)+o(\|h\|^{p}),
\label{eq:1.1}
\end{equation}
where
\begin{equation}
\Phi_{x,p}(y)\equiv f(x)+\sum_{i=1}^{p}\dfrac{1}{i!}D^{i}f(x)[y-x]^{i},\,\,y\in\E.
\label{eq:1.2}
\end{equation}
Since $D^{p}f(x)[\,.\,]$ is a 
symmetric $p$-linear form, its norm is defined as:
\begin{equation*}
\|D^{p}f(x)\|=\max_{h_{1},\ldots,h_{p}}\left\{\left|D^{p}f(x)[h_{1},\ldots,h_{p}]\right|\,:\,\|h_{i}\|\leq 1,\,i=1,\ldots,p\right\}.
\end{equation*}
It can be shown that (see, e.g., Appendix 1 in \cite{NES7})
\begin{equation*}
\|D^{p}f(x)\|=\max_{h}\left\{\left|D^{p}f(x)[h]^{p}\right|\,:\,\|h\|\leq 1\right\}.
\end{equation*}
Similarly, since $D^{p}f(x)[.\,,\ldots,\,.]-D^{p}f(y)[.,\ldots,.]$ is also a symmetric $p$-linear form for fixed $x,y\in\E$, it follows that
 \begin{equation*}
\|D^{p}f(x)-D^{p}f(y)\|=\max_{h}\left\{\left|D^{p}f(x)[h]^{p}-D^{p}f(y)[h]^{p}\right|\,:\,\|h\|\leq 1\right\}.
\end{equation*}

\section{Problem Statement}

In this paper we consider methods for solving the following minimization problem
\begin{equation}
\min_{x\in\E}\,f(x),
\label{eq:2.1}
\end{equation}
where $f:\E\to\R$ is a convex function $p$-times differentiable. We assume that (\ref{eq:2.1}) has at least one optimal solution $x^{*}\in\E$. As in \cite{GN3}, the level of smoothness of the objective $f$ will be characterized by the family of H\"{o}lder constants
\begin{equation}
H_{f,p}(\nu)\equiv\sup_{x,y\in\E}\left\{\dfrac{\|D^{p}f(x)-D^{p}f(y)\|}{\|x-y\|^{\nu}}\right\},\,\,0\leq \nu\leq 1.
\label{eq:2.3}
\end{equation}
From (\ref{eq:2.3}), it can be shown that, for all $x,y\in\E$, 
\begin{equation}
|f(y)-\Phi_{x,p}(y)|\leq\dfrac{H_{f,p}(\nu)}{p!}\|y-x\|^{p+\nu},
\label{eq:2.4}
\end{equation}
\begin{equation}
\|\nabla f(y)-\nabla\Phi_{x,p}(y)\|_{*}\leq\dfrac{H_{f,p}(\nu)}{(p-1)!}\|y-x\|^{p+\nu-1},
\label{eq:2.5}
\end{equation}
and
\begin{equation}
\|\nabla^{2}f(y)-\nabla^{2}\Phi_{x,p}(y)\|_{*}\leq\dfrac{H_{f,p}(\nu)}{(p-2)!}\|y-x\|^{p+\nu-2}.
\label{eq:ale1}
\end{equation}
Given $x\in\E$, if $H_{f,p}(\nu)<+\infty$ and $H\geq H_{f,p}(\nu)$, by (\ref{eq:2.4}) we have
\begin{equation}
f(y)\leq \Phi_{x,p}(y)+\dfrac{H}{p!}\|y-x\|^{p+\nu},\,\,y\in\E.
\label{eq:2.6}
\end{equation}
This property motivates the use of the following class of models of $f$ around $x\in\E$:
\begin{equation}
\Omega_{x,p,H}^{(\alpha)}(y)=\Phi_{x,p}(y)+\dfrac{H}{p!}\|y-x\|^{p+\alpha},\,\,\alpha\in [0,1].
\label{eq:2.7}
\end{equation}
Note that, by (\ref{eq:2.6}), if $H\geq H_{f,p}(\nu)$ then $f(y)\leq \Omega_{x,p,H}^{(\nu)}(y)$ for all $y\in\E$.

\section{Tensor Schemes Without Acceleration}

Let us consider the following assumption:
\begin{itemize}
\item[\textbf{H1}] $H_{f,p}(\nu)<+\infty$ for some $\nu\in [0,1]$.
\end{itemize}

\noindent Regarding the smoothness parameter $\nu$, there are only two possible situations: either $\nu$ is known, or $\nu$ is unknown. In order to cover both cases in a single framework, as in \cite{GN3}, we shall consider the parameter
\begin{equation}
\alpha=\left\{\begin{array}{ll} \nu,&\text{if}\,\,\nu\,\,\text{is known},\\
                                1, &\text{if}\,\,\nu\,\,\text{is unknown}.
              \end{array}
       \right.
\label{eq:3.1}
\end{equation}
\begin{mdframed}
\noindent\textbf{Algorithm 1. Tensor Method (Algorithm 2 in \cite{GN3})}
\\[0.2cm]
\noindent\textbf{Step 0.} Choose $x_{0}\in\E$, $H_{0}>0$, $\theta\geq 0$ and $\epsilon\in (0,1)$.  Set $\alpha$ by (\ref{eq:3.1}) and $t:=0$.\\
\noindent\textbf{Step 1.} If $\|\nabla f(x_{t})\|_{*}\leq\epsilon$, STOP.\\
\noindent\textbf{Step 2.} Set $i:=0$.\\
\noindent\textbf{Step 2.1} Compute an approximate solution $x_{t,i}^{+}$ to 
\begin{equation}
\min_{y\in\E}\,\Omega_{x_{t},p,2^{i}H_{t}}^{(\alpha)}(y)
\label{eq:3.2}
\end{equation}
such that
\begin{equation}
\Omega_{x_{t},p,2^{i}H_{t}}^{(\alpha)}(x_{t,i}^{+})\leq f(x_{t})\quad\text{and}\quad \|\nabla\Omega_{x_{t},p,2^{i}H_{t}}^{(\alpha)}(x_{t,i}^{+})\|_{*}\leq\theta\|x_{t,i}^{+}-x_{t}\|^{p+\alpha-1}.
\label{eq:3.3}
\end{equation}
\noindent\textbf{Step 2.2.} If either $\|\nabla f(x_{t,i}^{+})\|_{*}\leq\epsilon$ or 
\begin{equation} 
f(x_{t})-f(x_{t,i}^{+})\geq\dfrac{1}{8(p+1)!(2^{i}H_{t})^{\frac{1}{p+\alpha-1}}}\|\nabla f(x_{t,i}^{+})\|_{*}^{\frac{p+\alpha}{p+\alpha-1}},
\label{eq:3.5}
\end{equation}
holds, set $i_{t}:=i$ and go to Step 3. Otherwise, set $i:=i+1$ and go to Step 2.1.\\
\noindent\textbf{Step 3.} Set $x_{t+1}=x_{t,i_{t}}^{+}$ and $H_{t+1}=2^{i_{t}-1}H_{t}$.\\
\noindent\textbf{Step 4.} Set $t:=t+1$ and go back to Step 1.
\end{mdframed}

\begin{remark}
If $\nu$ is unknown, by (\ref{eq:3.1}) we set $\alpha=1$ in Algorithm 1. The resulting algorithm is a universal scheme that can be viewed as a generalization of the universal second-order method (6.10) in \cite{GN}. Moreover, it is worth mentioning that for $p=3$ and $\alpha=\nu$, one case use Gradient Methods with Bregman distance \cite{GN4} to approximately solve (\ref{eq:3.2}) in the sense of (\ref{eq:3.3}).
\end{remark}

For both cases ($\nu$ known or unknown), Algorithm 1 is a particular instance of Algorithm 1 in \cite{GN3} in which $M_{t}=2^{i_{t}}H_{t}$ for all $t\geq 0$. Let us define the following function of $\epsilon$:
\begin{equation}
N_{\nu}(\epsilon)=\left\{\begin{array}{ll} \max\left\{\dfrac{3H_{f,p}(\nu)}{2},3\theta (p-1)!\right\},&\text{if}\,\,\nu\,\,\text{is known},\\
                                          \max\left\{\theta,\left(\dfrac{3H_{f,p}(\nu)}{2}\right)^{\frac{p}{p+\nu-1}}4^{\frac{1-\nu}{p+\nu-1}}\right\}\epsilon^{-\frac{1-\nu}{p+\nu-1}},&\text{if}\,\,\nu\,\,\text{is unknown}.
                         \end{array}
                  \right.
\label{eq:3.6}
\end{equation}
The next lemma provides upper bounds on $M_{t}$ and on the number of calls of the oracle in Algorithm 1.
\begin{lemma}
\label{lem:3.1}
Suppose that H1 holds. Given $\epsilon\in (0,1)$, assume that $\left\{x_{t}\right\}_{t=0}^{T}$ is a sequence generated by Algorithm 1 such that
\begin{equation}
\|\nabla f(x_{t})\|_{*}>\epsilon,\quad t=0,\ldots,T.
\label{eq:3.7}
\end{equation}
Then,
\begin{equation}
H_{t}\leq\max\left\{H_{0},N_{\nu}(\epsilon)\right\},\quad\text{for}\,\,t=0,\ldots,T,
\label{eq:3.8}
\end{equation}
and, consequently, 
\begin{equation}
M_{t}=2^{i_{t}}H_{t}\leq 2\max\left\{H_{0},N_{\nu}(\epsilon)\right\},\quad\text{for}\,\,t=0,\ldots,T-1,
\label{eq:3.9}
\end{equation}

Moreover, the number $O_{T}$ of calls of the oracle after $T$ iterations is bounded as follows:
\begin{equation}
O_{T}\leq 2T+\log_{2}\max\left\{H_{0},N_{\nu}(\epsilon)\right\}-\log_{2}H_{0}.
\label{eq:3.10}
\end{equation}
\end{lemma}

\begin{proof}
Let us prove (\ref{eq:3.8}) by induction. Clearly it holds for $t=0$. Assume that (\ref{eq:3.8}) is true for some $t$, $0\leq t\leq T-1$. If $\nu$ is known, then by (\ref{eq:3.1}) we have $\alpha=\nu$. Thus, it follows from H1 and Lemma A.2 in \cite{GN3} that the final value of $2^{i_{t}}H_{t}$ cannot be bigger than 
$2\max\left\{(3/2)H_{f,p}(\nu),3\theta (p-1)!\right\}$, since otherwise we should stop the line search earlier. Therefore,
\begin{equation*}
H_{t+1}=\frac{1}{2}2^{i_{t}}H_{t}\leq\max\left\{\dfrac{3H_{f,p}(\nu)}{2},3\theta (p-1)!\right\}=N_{\nu}(\epsilon)\leq\max\left\{H_{0},N_{\nu}(\epsilon)\right\},
\end{equation*}
that is, (\ref{eq:3.8}) holds for $t=t+1$. On the other hand, if $\nu$ is unknown, we have $\alpha=1$. In view of  (\ref{eq:3.7}), Corollary A.5 \cite{GN3} and $\epsilon\in (0,1)$, we must have
\begin{equation*}
2^{i_{t}}H_{t}\leq 2\max\left\{\theta,\left(\dfrac{3H_{f,p}(\nu)}{2}\right)^{\frac{p}{p+\nu-1}}\left(\dfrac{4}{\epsilon}\right)^{\frac{1-\nu}{p+\nu-1}}\right\}\leq 2N_{\nu}(\epsilon).
\end{equation*}
Consequently, it follows that
\begin{equation*}
H_{t+1}=\frac{1}{2}2^{i_{t}}H_{t}\leq N_{\nu}(\epsilon)\leq\max\left\{H_{0},N_{\nu}(\epsilon)\right\},
\end{equation*}
that is, (\ref{eq:3.8}) holds for $t+1$. This completes the induction argument. Using (\ref{eq:3.8}), for $t=0,\ldots,T-1$ we get $M_{t}=2H_{t+1}\leq 2\max\left\{H_{0},N_{\nu}(\epsilon)\right\}$. Finally, note that at the $t$th iteration of Algorithm 1, the oracle is called $i_{t}+1$ times. Since $H_{t+1}=2^{i_{t}-1}H_{t}$, it follows that $i_{t}-1=\log_{2}H_{t+1}-\log_{2}H_{t}$. Thus, by (\ref{eq:3.8}) we get
\begin{eqnarray*}
O_{T}&=&\sum_{t=0}^{T-1}(i_{t}+1)=\sum_{t=0}^{T-1}\,2+\log_{2}H_{t+1}-\log_{2}H_{t}=2T+\log_{2}H_{T}-\log_{2}H_{0}\\
&\leq & 2T+\log_{2}\max\left\{H_{0},N_{\nu}(\epsilon)\right\}-\log_{2}H_{0},
\end{eqnarray*}
and the proof is complete.
\end{proof}

\noindent Let us consider the additional assumption:
\begin{itemize}
\item[\textbf{H2}] The level sets of $f$ are bounded, that is, $\max_{x\in \mathcal{L}(x_{0})}\|x-x^{*}\|\leq D_{0}\in (1,+\infty)$ for $\mathcal{L}(x_{0})\equiv\left\{x\in\E\,:\,f(x)\leq f(x_{0})\right\}$, with $x_{0}$ being the starting point.
\end{itemize}
The next theorem gives global convergence rates for Algorithm 1 in terms of the functional residual.
\begin{theorem}
\label{thm:june3.1}
Suppose that H1 and H2 are true and let $\left\{x_{t}\right\}_{t=0}^{T}$ be a sequence generated by Algorithm 1 such that, for $t=0,\ldots,T$, we have
\begin{equation*}
\|\nabla f(x_{t,i}^{+})\|_{*}>\epsilon,\quad i=0,\ldots,i_{t}.
\end{equation*}
Let $m$ be the first iteration number such that
\begin{equation*}
f(x_{m})-f(x^{*})\leq 4[8(p+1)!]^{p+\alpha-1}\max\left\{H_{0},N_{\nu}(\nu)\right\}D_{0}^{p+\alpha},
\end{equation*}
and assume that $m < T$. Then
\begin{equation}
m\leq\dfrac{1}{\ln\left(\frac{p+\alpha}{p+\alpha-1}\right)}\ln\max\left\{1,\log_{2}\dfrac{f(x_{0})-f(x^{*})}{2[8(p+1)!]^{p+\alpha-1}\max\left\{H_{0},N_{\nu}(\epsilon)\right\}D_{0}^{p+\alpha}}\right\}
\label{eq:june3.13}
\end{equation}
and, for all $k$, $m<k\leq T$, we have
\begin{equation}
f(x_{k})-f(x^{*})\leq\dfrac{[24p(p+1)!]^{p+\alpha-1}2\max\left\{H_{0},N_{\nu}(\epsilon)\right\}D_{0}^{p+\alpha}}{(k-m)^{p+\alpha-1}}.
\label{eq:june3.14}
\end{equation}
\end{theorem}

\begin{proof}
By Lemma \ref{lem:3.1}, this result follows from Theorem 3.1 in \cite{GN3} with $M_{\nu}=2\max\left\{H_{0},N_{\nu}(\epsilon)\right\}$.
\end{proof}

Now, we can derive global convergence rates for Algorithm 1 in terms of the norm of the gradient.
\begin{theorem}
\label{thm:main3.1}
Under the same assumptions of Theorem \ref{thm:june3.1}, if $T=m+3s$ for some $s\geq 1$, then
\begin{equation}
g_{T}^{*}\equiv\min_{0\leq t\leq T}\|\nabla f(x_{t})\|_{*}\leq 2\left[\dfrac{288p(p+1)!D_{0}}{T-m}\right]^{p+\alpha-1}\max\left\{H_{0},N_{\nu}(\epsilon)\right\}.
\label{eq:3.12}
\end{equation}
Consequently,
\begin{equation}
T< m+\kappa_{1}^{(\nu)}[288p(p+1)!D_{0}]\epsilon^{-\frac{1}{p+\nu-1}},
\label{eq:3.13}
\end{equation}  
with
\begin{equation*}
\kappa_{1}^{(\nu)}=\left\{\begin{array}{ll} \left(2\max\left\{H_{0},\dfrac{3H_{f,p}(\nu)}{2},3\theta (p-1)!\right\}\right)^{\frac{1}{p+\nu-1}},&\text{if}\,\,\nu\,\,\text{is known},\\
                                          \left(2\max\left\{H_{0},\theta,\left(\dfrac{3H_{f,p}(\nu)}{2}\right)^{\frac{p}{p+\nu-1}}4^{\frac{1-\nu}{p+\nu-1}}\right\}\right)^{\frac{1}{p}},&\text{if}\,\,\nu\,\,\text{is unknown}.
                         \end{array}
                  \right.
\end{equation*}
\end{theorem}

\begin{proof}
By Theorem \ref{thm:june3.1}, we have
\begin{equation}
f(x_{k})-f(x^{*})\leq \dfrac{[24p(p+1)!]^{p+\alpha-1}2\max\left\{H_{0},N_{\nu}(\epsilon)\right\}D_{0}^{p+\alpha}}{(k-m)^{p+\alpha-1}},
\label{eq:3.14}
\end{equation}
for all $k$, $m<k\leq T$. In particular, it follows from (\ref{eq:3.5}) and (\ref{eq:3.14}) that
\begin{eqnarray*}
\dfrac{[24p(p+1)!]^{p+\alpha-1}2\max\left\{H_{0},N_{\nu}(\epsilon)\right\}D_{0}^{p+\alpha}}{(2s)^{p+\alpha-1}}&\geq & f(x_{m+2s})-f(x^{*})\\
           & = & f(x_{T})-f(x^{*})+\sum_{k=m+2s}^{T-1}(f(x_{k})-f(x_{k+1}))\\
           &\geq & \dfrac{s}{8(p+1)!}\left[\dfrac{1}{2\max\left\{H_{0},N_{\nu}(\epsilon)\right\}}\right]^{\frac{1}{p+\alpha-1}}(g_{T}^{*})^{\frac{p+\alpha}{p+\alpha-1}}.
\end{eqnarray*}
Therefore,
\begin{eqnarray*}
(g_{T}^{*})^{\frac{p+\alpha}{p+\alpha-1}}&\leq & \dfrac{8(p+1)![24p(p+1)!]^{p+\alpha-1}\left(2\max\left\{H_{0},N_{\nu}(\epsilon)\right\}\right)^{\frac{p+\alpha}{p+\alpha-1}}D_{0}^{p+\alpha}}{2^{p+\alpha-1}s^{p+\alpha}}\\
& \leq & \dfrac{(96p)^{p+\alpha-1}3^{p+\alpha}[(p+1)!]^{p+\alpha}\left(2\max\left\{H_{0},N_{\nu}(\epsilon)\right\}\right)^{\frac{p+\alpha}{p+\alpha-1}}D_{0}^{p+\alpha}}{(T-m)^{p+\alpha}},
\end{eqnarray*}
and so (\ref{eq:3.12}) holds. By assumption, we have $g_{T}^{*}>\epsilon$. Thus, by (\ref{eq:3.12}) we get
\begin{equation*}
\epsilon<2\left(\dfrac{288p(p+1)!D_{0}}{T-m}\right)^{p+\alpha-1}\max\left\{H_{0},N_{\nu}(\epsilon)\right\}
\end{equation*}
\begin{equation}
\Longrightarrow T<m+[288p(p+1)!D_{0}]\left(2\max\left\{H_{0},N_{\nu}(\epsilon)\right\}\right)^{\frac{1}{p+\alpha-1}}\epsilon^{-\frac{1}{p+\alpha-1}}.
\label{eq:3.15}
\end{equation}
Finally, by analyzing separately the cases in which $\nu$ is known and unknown, it follows from (\ref{eq:3.15}) and (\ref{eq:3.6}) that (\ref{eq:3.13}) is true.
\end{proof}

\begin{remark}
\label{rem:ale1}
Suppose that the objective $f$ in (\ref{eq:2.1}) is nonconvex and bounded from below by $f^{*}$. Then, it follows from (\ref{eq:3.5}) and (\ref{eq:3.9}) that
\begin{equation*}
f(x_{t})-f(x_{t+1})\geq\dfrac{1}{8(p+1)!}\left[\dfrac{1}{2\max\left\{H_{0},N_{\nu}(\epsilon)\right\}}\right]^{\frac{1}{p+\alpha-1}}\epsilon^{\frac{p+\alpha}{p+\alpha-1}},\,\,t=0,\ldots,T-1.
\end{equation*}
Summing up these inequalities, we get
\begin{equation*}
f(x_{0})-f^{*}\geq f(x_{0})-f(x_{T})\geq\dfrac{T}{8(p+1)!}\left[\dfrac{1}{2\max\left\{H_{0},N_{\nu}(\epsilon)\right\}}\right]^{\frac{1}{p+\alpha-1}}\epsilon^{\frac{p+\alpha}{p+\alpha-1}}
\end{equation*}
and so, by (\ref{eq:3.6}), we obtain $T\leq\mathcal{O}\left(\epsilon^{-\frac{p+\nu}{p+\nu-1}}\right)$.
This bound generalizes the bound of $\mathcal{O}\left(\epsilon^{-\frac{2+\nu}{1+\nu}}\right)$ proved in \cite{GN} for $p=2$. It agrees in order with the complexity bounds proved in \cite{Mart} and \cite{CGT2} for different universal tensor methods.
\end{remark}

\section{Accelerated Tensor Schemes}

The schemes presented here generalize the procedures described in \cite{NES8} for $p=1$ and $p=2$. Specifically, our general scheme is obtained by adding Step 2 of Algorithm 1 at the end of Algorithm 4 in \cite{GN3}, in order to relate the functional decrease with the norm of the gradient of $f$ in suitable points:
\begin{mdframed}
\noindent\textbf{Algorithm 2. Adaptive Accelerated Tensor Method}
\\[0.2cm]
\noindent\textbf{Step 0.} Choose $x_{0}\in\E$, $\tilde{H}_{0},H_{0}>0$, $\theta\geq 0$ and $\epsilon\in (0,1)$. Set $\alpha$ by (\ref{eq:3.1}) and define function $\psi_{0}(x)=\frac{1}{p+\alpha}\|x-x_{0}\|^{p+\alpha}$. Set $v_{0}=z_{0}=x_{0}$, $A_{0}=0$ and $t:=0$.\\
\noindent\textbf{Step 1.} If $\min\left\{\|\nabla f(x_{t})\|_{*},\|\nabla f(z_{t})\|_{*}\right\}\leq\epsilon$, STOP.\\
\noindent\textbf{Step 2.} Set $i:=0$.\\
\noindent\textbf{Step 2.1.} Compute the coefficient $a_{t,i}>0$ by solving equation 
\begin{equation*}
a_{t,i}^{p+\alpha}=\dfrac{1}{2^{(3p-1)}}\left[\dfrac{(p-1)!}{2^i\tilde{H}_{t}}\right](A_{t}+a_{t,i})^{p+\alpha-1}.
\end{equation*}
\noindent\textbf{Step 2.2.} Set $\gamma_{t,i}=\dfrac{a_{t,i}}{A_{t}+a_{t,i}}$ and compute $y_{t,i}=(1-\gamma_{t,i})x_{t}+\gamma_{t,i}v_{t}$.\\
\noindent\textbf{Step 2.3} Compute an approximate solution $x_{t,i}^{+}$ to $\min_{x\in\E}\,\Omega_{y_{t,i},p,2^{i}\tilde{H}_{t}}^{(\alpha)}(x)$, such that
\begin{equation*}
\Omega_{y_{t,i},p,2^{i}\tilde{H}_{t}}^{(\alpha)}(x_{t,i}^{+})\leq f(y_{t,i})
\quad\text{and}\quad\|\nabla\Omega_{y_{i,t},p,2^{i}\tilde{H}_{t}}^{(\alpha)}(x_{t,i}^{+})\|_{*}\leq\theta\|x_{t,i}^{+}-y_{t,i}\|^{p+\alpha-1}.
\end{equation*}
\noindent\textbf{Step 2.4.} If either condition $\|\nabla f(x_{t,i}^{+})\|_{*}\leq\epsilon$ or 
\begin{equation*}
\la \nabla f(x_{t,i}^{+}),y_{t,i}-x_{t,i}^{+}\ra\geq\dfrac{1}{4}\left[\dfrac{(p-1)!}{2^{i}\tilde{H}_{t}}\right]^{\frac{1}{p+\alpha-1}}\|\nabla f(x_{t,i}^{+})\|_{*}^{\frac{p+\alpha}{p+\alpha-1}}
\end{equation*}
holds, set $i_{t}:=i$ and go to Step 3. Otherwise, set $i:=i+1$ and go back to Step 2.1.\\
\noindent\textbf{Step 3.} Set $x_{t+1}=x_{t,i_{t}}^{+}$ and $\tilde{H}_{t+1}=2^{i_{t}-1}\tilde{H}_{t}$.\\
\noindent\textbf{Step 4.} Define $\psi_{t+1}(x)=\psi_{t}(x)+a_{t}\left[f(x_{t+1})+\langle\nabla f(x_{t+1}),x-x_{t+1}\rangle\right]$ and compute $v_{t+1}=\arg\min_{x\in\E}\,\psi_{t+1}(x)$.\\
\noindent\textbf{Step 5.} Set $\bar{z}_{t}=\arg\min\left\{f(y)\,:\,y\in\left\{z_{t},x_{t+1}\right\}\right\}$ and $j:=0$.\\
\noindent\textbf{Step 6.} Set $j:=0$.\\
\noindent\textbf{Step 6.1} Compute an approximate solution $z_{t,j}^{+}$ to $\min_{y\in\E}\,\Omega_{\bar{z}_{t},p,2^{j}H_{t}}^{(\alpha)}(y)$ such that
\begin{equation*}
\Omega_{\bar{z}_{t},p,2^{j}H_{t}}^{(\alpha)}(z_{t,j}^{+})\leq f(\bar{z}_{t})
\quad\text{and}\quad\|\nabla\Omega_{\bar{z}_{t},p,2^{j}H_{t}}^{(\alpha)}(z_{t,j}^{+})\|_{*}\leq\theta\|z_{t,j}^{+}-\bar{z}_{t}\|^{p+\alpha-1}.
\end{equation*}
\noindent\textbf{Step 6.2} If either $\|\nabla f(z_{t,j}^{+})\|_{*}\leq\epsilon$ or 
\begin{equation}
f(\bar{z}_{t})-f(z_{t,j}^{+})\geq\dfrac{1}{8(p+1)!(2^{j}H_{t})^{\frac{1}{p+\alpha-1}}}\|\nabla f(z_{t,j}^{+})\|_{*}^{\frac{p+\alpha}{p+\alpha-1}}
\label{eq:4.1}
\end{equation} 
holds, set $j_{t}:=j$ and go to Step 7. Otherwise, set $j:=j+1$ and go to Step 6.1.\\
\noindent\textbf{Step 7.} Set $z_{t+1}=z_{t,j_{t}}^{+}$, $H_{t+1}=2^{j_{t}-1}H_{t}$, $t:=t+1$ and go to Step 1.
\end{mdframed}
Let us define the following function of $\epsilon$: 
\begin{equation}
\tilde{N}_{\nu}(\epsilon)=\left\{\begin{array}{ll} (p+\nu-1)(H_{f,p}(\nu)+\theta(p-1)!),&\text{if}\,\,\nu\,\,\text{is known},\\
\max\left\{4\theta(p-1)!,(4H_{f,p}(\nu))^{\frac{p}{p+\nu-1}}\left(\dfrac{4}{\epsilon}\right)^{\frac{1-\nu}{p+\nu-1}}\right\},&\text{if}\,\,\nu\,\,\text{is unknown}.
                                    \end{array}
                             \right.
\label{eq:4.21}
\end{equation}

In Algorithm 2, note that $\left\{x_{t}\right\}$ is independent of $\left\{z_{t}\right\}$. The next theorem esta\-blishes global convergence rates for the functional residual with respect to $\left\{x_{t}\right\}$.
\begin{theorem}
\label{thm:B4.2} Assume that H1 holds and let the sequence
$\left\{x_{t}\right\}_{t=0}^{T}$ be generated by
Algorithm 2 such that, for $t=0,\ldots,T$ we have
\begin{equation}
\|\nabla f(x_{t,i}^{+})\|_{*}>\epsilon,\quad i=0,\ldots,i_{t}.
\label{eq:june4.12}
\end{equation}
Then, 
\begin{equation}
f(x_{t})-f(x^{*})\leq\dfrac{2^{3p}\max\left\{\tilde{H}_{0},\tilde{N}_{\nu}(\epsilon)\right\}(p+\alpha)^{p+\alpha-1}\|x_{0}-x^{*}\|^{p+\alpha}}{(p-1)!(t-1)^{p+\alpha}},
\label{eq:june4.13}
\end{equation}
for $t=2,\ldots,T$.
\end{theorem}

\begin{proof}
As in the proof of Lemma \ref{lem:3.1}, it follows from (\ref{eq:june4.12}), (\ref{eq:4.21}) and Lemmas A.6 and A.7 in \cite{GN3} that
\begin{equation*}
\tilde{H}_{t}\leq\max\left\{\tilde{H}_{0},\tilde{N}_{\nu}(\epsilon)\right\},\,\,t=0,\ldots, T,
\end{equation*} 
which gives
\begin{equation*}
\tilde{M}_{t}=2^{i_{t}}\tilde{H}_{t}=2\left(2^{i_{t}}\tilde{H}_{t}\right)=2\tilde{H}_{t+1}\leq 2\max\left\{\tilde{H}_{0},\tilde{N}_{\nu}(\epsilon)\right\},\,\,t=0,\ldots,T-1.
\end{equation*}
Then, (\ref{eq:june4.13}) follows directly from Theorem 4.2 in \cite{GN3} with $M_{\nu}=2\max\left\{\tilde{H}_{0},\tilde{N}_{\nu}(\epsilon)\right\}$.
\end{proof}

Now we can obtain global convergence rates for Algorithm 2 in terms of the norm of the gradient.
\begin{theorem}
\label{thm:5.1}
Suppose that H1 holds and let sequences $\left\{x_{t}\right\}_{t=0}^{T}$ and $\left\{z_{t}\right\}_{t=0}^{T}$ be generated by Algorithm 2. Assume that, for $t=0,\ldots,T$, we have
\begin{equation}
\min\left\{\|\nabla f(x_{t,i}^{+})\|_{*},\|\nabla f(z_{t,j}^{+})\|_{*}\right\}>\epsilon,\quad i=0,\ldots,i_{t},\,\,j=0,\ldots,j_{t}.
\label{eq:5.6}
\end{equation}
If $T=2s$ for some $s>1$, then
\begin{equation}
g_{T}^{*}\equiv\min_{0\leq k\leq T}\|\nabla f(z_{t})\|_{*}\leq C_{\nu}(\epsilon)\|x_{0}-x^{*}\|^{p+\alpha-1}\left[\dfrac{p+1}{T-2}\right]^{\frac{(p+\alpha-1)(p+\alpha+1)}{p+\alpha}}
\label{eq:5.8}
\end{equation}
where
\begin{equation*}
C_{\nu}(\epsilon)=\left[2^{(4p+6)}\max\left\{\tilde{H}_{0},\tilde{N}_{\nu}(\epsilon)\right\}\max\left\{H_{0},N_{\nu}(\epsilon)\right\}^{\frac{1}{p+\alpha-1}}\right]^{\frac{p+\alpha-1}{p+\alpha}},
\end{equation*}
with $N_{\nu}(\epsilon)$ and $\tilde{N}_{\nu}(\epsilon)$ defined in (\ref{eq:3.6}) and (\ref{eq:4.21}), respectively. Consequently,
\begin{equation}
T\leq 2+2^{(4p+6)}(p+1)\max\left\{1,\tilde{H}_{0},H_{0},3p\left(H_{f,p}(\nu)+\theta (p-1)!\right)\right\}\|x_{0}-x^{*}\|^{\frac{p+\nu}{p+\nu+1}}\epsilon^{-\frac{p+\nu}{(p+\nu-1)(p+\nu+1)}},
\label{eq:5.9}
\end{equation}
if $\nu$ is known (i.e., $\alpha=\nu$), and 
\begin{equation}
T< 2+ 2^{(2p+3)}(p+1)\max\left\{1,\tilde{H}_{0},H_{0},4\left[(4H_{f,p}(\nu))^{\frac{p}{p+\nu-1}}+\theta (p-1)!\right]\right\}^{\frac{1}{2}}\|x_{0}-x^{*}\|^{\frac{p+1}{p+2}}\epsilon^{-\frac{p+1}{(p+\nu-1)(p+2)}},
\label{eq:june5.10}
\end{equation}
if $\nu$ is unknown (i.e., $\alpha=1$).
\end{theorem}

\begin{proof}
By Theorem \ref{thm:B4.2}, we have
\begin{equation}
f(x_{t})-f(x^{*})\leq\dfrac{2^{3p}\max\left\{\tilde{H}_{0},\tilde{N}_{\nu}(\epsilon)\right\}(p+\alpha)^{p+\alpha-1}\|x_{0}-x^{*}\|^{p+\alpha}}{(p-1)!(t-1)^{p+\alpha}}
\label{eq:5.12}
\end{equation} 
for $t=2,\ldots,T$.
On the other hand, as in Lemma \ref{lem:3.1}, by (\ref{eq:5.6}) we get 
\begin{equation*}
2^{j_{t}}H_{t}\leq2\max\left\{H_{0},N_{\nu}(\epsilon)\right\},\quad t=0,\ldots,T-1,
\end{equation*}
where $N_{\nu}(\epsilon)$ is defined in (\ref{eq:3.6}). Then, in view of (\ref{eq:4.1}), it follows that
\begin{equation}
f(\bar{z}_{t})-f(z_{t+1})\geq\dfrac{1}{8(p+1)!}\left[\dfrac{1}{2\max\left\{H_{0},N_{\nu}(\epsilon)\right\}}\right]^{\frac{1}{p+\alpha-1}}\|\nabla f(z_{t+1})\|_{*}^{\frac{p+\alpha}{p+\alpha-1}},
\label{eq:5.13}
\end{equation}
for $t=0,\ldots,T-1$. In particular, $f(z_{t+1})\leq f(\bar{z}_{t})$ for $t=0,\ldots,T-1$. Moreover, by the definition of $\bar{z}_{t}$, we get $f(\bar{z}_{t})\leq f(x_{t+1})$ and $f(\bar{z}_{t})\leq f(z_{t})$. Therefore
\begin{equation}
f(z_{t})\leq f(x_{t}),\quad t=0,\ldots,T,
\label{eq:5.14}
\end{equation}
and
\begin{equation}
f(z_{t})-f(z_{t+1})\geq\dfrac{1}{8(p+1)!}\left[\dfrac{1}{2\max\left\{H_{0},N_{\nu}(\epsilon)\right\}}\right]^{\frac{1}{p+\alpha-1}}\|\nabla f(z_{t+1})\|_{*}^{\frac{p+\alpha}{p+\alpha-1}}.
\label{eq:set5.15}
\end{equation}
Now, since $T=2s$, summing up (\ref{eq:set5.15}), we get
\begin{eqnarray*}
\dfrac{2^{3p}\max\left\{\tilde{H}_{0},\tilde{N}_{\nu}(\epsilon)\right\}(p+\alpha)^{p+\alpha-1}\|x_{0}-x^{*}\|^{p+\alpha}}{(p-1)!(s-1)^{p+\alpha}}&\geq & f(x_{s})-f(x^{*})\geq  f(z_{s})-f(x^{*})\\
           & = & f(z_{T})-f(x^{*})+\sum_{k=s}^{T-1}(f(z_{k})-f(z_{k+1}))\\
           &\geq & \dfrac{(s-1)}{8(p+1)!}\left[\dfrac{1}{2\max\left\{H_{0},N_{\nu}(\epsilon)\right\}}\right]^{\frac{1}{p+\alpha-1}}(g_{T}^{*})^{\frac{p+\alpha}{p+\alpha-1}}.
\end{eqnarray*}
Thus,
\begin{equation*}
(g_{T}^{*})^{\frac{p+\alpha}{p+\alpha-1}}\leq \dfrac{2^{(4p+6)}\max\left\{\tilde{H}_{0},\tilde{N}_{\nu}(\epsilon)\right\}\max\left\{H_{0},N_{\nu}(\epsilon)\right\}^{\frac{1}{p+\alpha-1}}(p+1)^{p+\alpha+1}\|x_{0}-x^{*}\|^{p+\alpha}}{(T-2)^{p+\alpha+1}},
\end{equation*}
and so (\ref{eq:5.8}) holds. By assumption, we have $g_{T}^{*}>\epsilon$. Thus, it follows from (\ref{eq:5.8}) that
\begin{equation*}
\epsilon<C_{\nu}(\epsilon)\|x_{0}-x^{*}\|^{p+\alpha-1}\left[\dfrac{p+1}{T-2}\right]^{\frac{(p+\alpha-1)(p+\alpha+1)}{(p+\alpha)}}
\end{equation*}
\begin{equation}
\Longrightarrow T<2+(p+1)\left[\dfrac{C_{\nu}(\epsilon)}{\epsilon}\right]^{\frac{p+\alpha}{(p+\alpha-1)(p+\alpha+1)}}\|x_{0}-x^{*}\|^{\frac{p+\alpha}{p+\alpha+1}}.
\label{eq:june5.16}
\end{equation}
If $\nu$ is known, by (\ref{eq:3.6}) and (\ref{eq:4.21}) we have $\max\left\{N_{\nu}(\epsilon),\tilde{N}_{\nu}(\epsilon)\right\}\leq 3p\left(H_{f,p}(\nu)+\theta (p-1)!\right)$.
Then,
\begin{equation*}
\max\left\{\tilde{H}_{0},\tilde{N}_{\nu}(\epsilon)\right\}\max\left\{H_{0},N_{\nu}(\epsilon)\right\}^{\frac{1}{p+\nu-1}}\leq\max\left\{\tilde{H}_{0},H_{0},3p\left(H_{f,p}(\nu)+\theta (p-1)!\right)\right\}^{\frac{p+\nu}{p+\nu-1}},
\end{equation*}
and so
\begin{equation}
C_{\nu}(\epsilon)\leq 2^{(4p+6)}\max\left\{\tilde{H}_{0},H_{0},3p\left(H_{f,p}(\nu)+\theta (p-1)!\right)\right\}.
\label{eq:june5.17}
\end{equation}
Combining (\ref{eq:june5.16}), (\ref{eq:june5.17}) and $\frac{p+\nu}{(p+\nu-1)(p+\nu+1)}\leq 1$, we obtain (\ref{eq:5.9}). If $\nu$ is unknown, it follows from (\ref{eq:3.6}) and (\ref{eq:4.21}) that
\begin{equation*}
\max\left\{N_{\nu}(\epsilon),\tilde{N}_{\nu}(\epsilon)\right\}\leq 4\left[(4H_{f,p}(\nu))^{\frac{p}{p+\nu-1}}+\theta (p-1)!\right]\epsilon^{-\frac{1-\nu}{p+\nu-1}}.
\end{equation*}
Then,
\begin{equation*}
\max\left\{\tilde{H}_{0},\tilde{N}_{\nu}(\epsilon)\right\}\max\left\{H_{0},N_{\nu}(\epsilon)\right\}^{\frac{1}{p}}\leq\left[\max\left\{\tilde{H}_{0},H_{0},4\left[(4H_{f,p}(\nu))^{\frac{p}{p+\nu-1}}+\theta (p-1)!\right]\right\}\epsilon^{-\frac{1-\nu}{p+\nu-1}}\right]^{\frac{p+1}{p}},
\end{equation*}
and so
\begin{equation}
C_{\nu}(\epsilon)\leq2^{(4p+6)}\max\left\{\tilde{H}_{0},H_{0},4\left[(4H_{f,p}(\nu))^{\frac{p}{p+\nu-1}}+\theta (p-1)!\right]\right\}\epsilon^{-\frac{1-\nu}{p+\nu-1}}.
\label{eq:june5.18}
\end{equation}
Combining (\ref{eq:june5.16}), (\ref{eq:june5.18}) and $\frac{p+1}{p(p+2)}<\frac{1}{2}$, we obtain (\ref{eq:june5.10}).
\end{proof}

\begin{remark}
When $\nu=1$, bounds (\ref{eq:5.9}) and (\ref{eq:june5.10}) have the same dependence on $\epsilon$. However, when $\nu\neq 1$, the bound of $\mathcal{O}\left(\epsilon^{-(p+1)/[(p+\nu-1)(p+2)]}\right)$ obtained for the universal scheme (i.e., $\alpha=1$) is worse than the bound of $\mathcal{O}\left(\epsilon^{-(p+\nu)/[(p+\nu-1)(p+\nu+1)]}\right)$ obtained for the non-universal scheme (i.e., $\alpha=\nu$). In both cases, these complexity bounds are better than the bound of $\mathcal{O}\left(\epsilon^{-1/(p+\nu-1)}\right)$ proved for Algorithm 1.
\end{remark}

\section{Composite Minimization}

From now on, we will assume that $\nu$ and $H_{f,p}(\nu)$ are known. In this setting, we can consider the composite minimization problem:
\begin{equation}
\min_{x\in\E}\,\tilde{f}(x)\equiv f(x)+\varphi(x),
\label{eq:ale5.1}
\end{equation}
where $f:\E\to\mathbb{R}$ is a convex function satisfying H1 (see page 4), and $\varphi:\E\to\mathbb{R}\cup\left\{+\infty\right\}$ is a simple closed convex function whose effective domain has nonempty relative interior, that is, $\text{ri}\left(\dom\varphi\right)\neq \emptyset$. We assume that there exists at least one optimal solution $x^{*}\in\E$ for (\ref{eq:ale5.1}). By (\ref{eq:2.4}), if $H\geq H_{f,p}(\nu)$ we have
\begin{equation*}
\tilde{f}(y)\leq \Omega_{x,p,H}^{(\nu)}(y)+\varphi(y),\quad\forall y\in\E.
\end{equation*}
This motivates the following class of models of $\tilde{f}$ around a fixed point $x\in\E$:
\begin{equation}
\tilde{\Omega}_{x,p,H}^{(\nu)}(y)\equiv \Omega_{x,p,H}^{(\nu)}(y)+\varphi(y), 
\label{eq:ale5.2}
\end{equation}
where $\Omega_{x,p,H}^{(\nu)}(\,.\,)$ is defined in (\ref{eq:2.7}). The next lemma gives a sufficient condition for function $\Omega_{x,p,H}^{(\nu)}(\,.\,)$ to be convex. Its proof is an adaptation of the proof of Theorem 1 in \cite{NES6}.

\begin{lemma}
\label{lem:ale5.1}
Suppose that H1 holds for some $p\geq 2$. Then, for any $x,y\in\E$ we have
\begin{equation}
\nabla^{2}f(y)\preceq \nabla^{2}\Phi_{x,p}(y)+\dfrac{H_{f,p}(\nu)}{(p-2)!}\|y-x\|^{p+\nu-2}B. 
\label{eq:ale5.3}
\end{equation}
Moreover, if $H\geq (p-1)H_{f,p}(\nu)$, then function $\Omega_{x,p,H}^{(\nu)}(\,.\,)$ is convex for any $x\in\E$.
\end{lemma}

\begin{proof}
For any $u\in\E$, it follows from (\ref{eq:ale1}) that 
\begin{equation*}
\la\left(\nabla^{2}f(y)-\nabla^{2}\Phi_{x,p}(y)\right)u,u\ra \leq  \|\nabla^{2}f(y)-\nabla^{2}\Phi_{x,p}(y)\|_{*}\|u\|^{2}\leq \dfrac{H_{f,p}(\nu)}{(p-2)!}\|y-x\|^{p+\nu-2}\|u\|^{2}.
\end{equation*}
Since $u\in\E$ is arbitrary, we get (\ref{eq:ale5.3}).

Now, suppose that $H\geq (p-1)H_{f,p}(\nu)$. Then, by (\ref{eq:ale5.3}) we have
\begin{eqnarray*}
0\preceq\nabla^{2}f(y)& \preceq &\nabla^{2}\Phi_{x,p}(y)+\dfrac{H_{f,p}(\nu)(p-1)}{(p-1)!}\|y-x\|^{p+\nu-2}B\\
                      &\preceq &\nabla^{2}\Phi_{x,p}(y)+\dfrac{H}{(p-1)!}\|y-x\|^{p+\nu-2}B\\
                      &\preceq & \nabla^{2}\Phi_{x,p}(y)+\dfrac{H(p+\nu)}{p!}\|y-x\|^{p+\nu-2}B\\
                      &\preceq & \nabla^{2}\Phi_{x,p}(y)+\nabla^{2}\left(\dfrac{H}{p!}\|y-x\|^{p+\nu}\right)= \nabla^{2}\Omega_{x,p,H}^{(\nu)}(y).
\end{eqnarray*}
Therefore, $\Omega_{x,p,H}^{(\nu)}(y)$ is convex.
\end{proof}

From Lemma \ref{lem:ale5.1}, if $H\geq (p-1)H_{f,p}(\nu)$ it follows that $\tilde{\Omega}_{x,p,H}^{(\nu)}(\,.\,)$ is also convex. In this case, since $\text{ri}\left(\dom\varphi\right)\neq\emptyset$, any solution $x^{+}$ of
\begin{equation}
\min_{y\in\E}\tilde{\Omega}_{x,p,H}^{(\nu)}(y)
\label{eq:aug5.3}
\end{equation}
satisfies the first-order optimality condition:
\begin{equation}
0\in\partial\tilde{\Omega}_{x,p,H}^{(\nu)}(x^{+})=\partial\Omega_{x,p,H}^{(\nu)}(x^{+})+\partial\varphi(x^{+})=\left\{\nabla\Omega_{x,p,H}^{(\nu)}(x^{+})\right\}+\partial\varphi (x^{+}).
\label{eq:ale5.4}
\end{equation}
Therefore, there exists $g_{\varphi}(x^{+})\in\partial\varphi(x^{+})$ such that
\begin{equation}
\nabla\Omega_{x,p,H}^{(\nu)}(x^{+})+g_{\varphi}(x^{+})=0.
\label{eq:ale5.5}
\end{equation}
Instead of solving (\ref{eq:aug5.3}) exactly, in our algorithms we consider inexact solutions $x^{+}$ such that\footnote{Conditions (\ref{eq:aug5.6}) have already been used in \cite{JIA} and are the composite analogue of the conditions proposed in \cite{Birgin}. It is worth to mention that, for $p=3$ and $\nu=1$, the tensor model $\Omega_{x,p,M}^{(\nu)}(\,.\,)$ has very nice relative smoothness properties (see \cite{NES6}) which allow the approximate solution of (\ref{eq:aug5.3}) by Bregman Proximal Gradient Algorithms \cite{LU,BOL}.}
\begin{equation}
\tilde{\Omega}_{x,p,H}^{(\nu)}(x^{+})\leq\tilde{f}(x)\quad\text{and}\quad \|\nabla\Omega_{x,p,H}^{(\nu)}(x^{+})+g_{\varphi}(x^{+})\|_{*}\leq\theta\|x^{+}-x\|^{p+\nu-1},
\label{eq:aug5.6}
\end{equation}
for some $g_{\varphi}(x^{+})\in\partial\varphi(x^{+})$ and $\theta\geq 0$. For such points $x^{+}$, we define
\begin{equation}
\nabla\tilde{f}(x^{+})\equiv \nabla f(x^{+})+g_{\varphi}(x^{+}),
\label{eq:ale5.7}
\end{equation}
with $g_{\varphi}(x^{+})$ satisfying (\ref{eq:aug5.6}). Clearly, we have $\nabla\tilde{f}(x^{+})\in\partial\tilde{f}(x^{+})$. 
\begin{lemma}
\label{lem:ale5.2}
Suppose that H1 holds and let $x^{+}$ be an approximate solution of (\ref{eq:aug5.3}) such that (\ref{eq:aug5.6}) holds for some $x\in\E$. If 
\begin{equation}
H\geq\max\left\{pH_{f,p}(\nu),3\theta (p-1)!\right\},
\label{eq:aug5.7}
\end{equation}
then
\begin{equation}
\tilde{f}(x)-\tilde{f}(x^{+})\geq\dfrac{1}{8(p+1)!H^{\frac{1}{p+\nu-1}}}\|\nabla\tilde{f}(x^{+})\|_{*}^{\frac{p+\nu}{p+\nu-1}}.
\label{eq:ale5.9}
\end{equation}
\end{lemma}

\begin{proof}
By (\ref{eq:ale5.7}), (\ref{eq:2.5}), (\ref{eq:2.7}), (\ref{eq:aug5.6}) and (\ref{eq:aug5.7}) we have
\begin{eqnarray}
\|\nabla\tilde{f}(x^{+})\|_{*}& = & \|\nabla f(x^{+})+g_{\varphi}(x^{+})\|_{*}\nonumber\\
                              &\leq & \|\nabla f(x^{+})-\nabla\Phi_{x,p}(x^{+})\|_{*}+\|\nabla\Phi_{x,p}(x^{+})-\nabla\Omega_{x,p,H}^{(\nu)}(x^{+})\|_{*}\nonumber\\
                              &  & +\|\nabla\Omega_{x,p,H}^{(\nu)}(x^{+})+g_{\varphi}(x^{+})\|_{*}\nonumber\\
                              &\leq &\left[\dfrac{H_{f,p}(\nu)}{(p-1)!}+\dfrac{H(p+\nu)}{p!}+\theta\right]\|x^{+}-x\|^{p+\nu-1}\leq  2H\|x^{+}-x\|^{p+\nu-1},
\label{eq:ale5.10}
\end{eqnarray}
where the last inequality is due to $p\geq 2$. On the other hand, by (\ref{eq:2.4}), (\ref{eq:ale5.2}), (\ref{eq:aug5.7}), we have
\begin{eqnarray*}
\tilde{f}(x^{+})&\leq &\tilde{\Omega}_{x,p,H_{f,p}(\nu)}^{(\nu)}(x^{+})=\Phi_{x,p}(x^{+})+\dfrac{H_{f,p}(\nu)}{p!}\|x^{+}-x\|^{p+\nu}+\varphi(x^{+})\\
                & = & \Phi_{x,p}(x^{+})+\dfrac{H}{p!}\|x^{+}-x\|^{p+\nu}-\dfrac{(H-H_{f,p}(\nu))}{p!}\|x^{+}-x\|^{p+\nu}+\varphi(x^{+})\\
                & = & \tilde{\Omega}_{x,p,H}^{(\nu)}(x^{+})-\dfrac{H-H_{f,p}(\nu)}{p!}\|x^{+}-x\|^{p+\nu}\leq \tilde{f}(x^{+})-\dfrac{H-H_{f,p}(\nu)}{p!}\|x^{+}-x\|^{p+\nu}.
\end{eqnarray*}
Note that $H\geq pH_{f,p}(\nu)\geq\left(\frac{p+1}{p}\right)H_{f,p}(\nu)$. Thus,
\begin{eqnarray}
\tilde{f}(x)-\tilde{f}(x^{+})&\geq&\dfrac{H-H_{f,p}(\nu)}{p!}\|x^{+}-x\|^{p+\nu}\geq\dfrac{\left(H-\frac{1}{p+1}H\right)}{p!}\|x^{+}-x\|^{p+\nu}\nonumber\\
&=&\dfrac{H}{(p+1)!}\|x^{+}-x\|^{p+\nu}.
\label{eq:ale5.11}
\end{eqnarray}
Finally, combining (\ref{eq:ale5.10}) and (\ref{eq:ale5.11}), we get (\ref{eq:ale5.9}).
\end{proof}

In this composite context, let us consider the following scheme:
\begin{mdframed}
\noindent\textbf{Algorithm 3. Tensor Method for Composite Minimization}
\\[0.2cm]
\noindent\textbf{Step 0.} Choose $x_{0}\in\E$ and $\theta\geq 0$. Set $M=\max\left\{pH_{f,p}(\nu),3\theta (p-1)!\right\}$ and $t:=0$.\\
\noindent\textbf{Step 1.} Compute an approximate solution $x_{t+1}$ to $\min_{y\in\E}\tilde{\Omega}_{x_{t},p,M}^{(\nu)}(y)$ such that
\begin{equation*}
\tilde{\Omega}_{x_{t},p,M}^{(\nu)}(x_{t+1})\leq\tilde{f}(x_{t})\quad\text{and}\quad \|\nabla\Omega_{x_{t},p,M}^{(\nu)}(x_{t+1})+g_{\varphi}(x_{t+1})\|_{*}\leq\theta\|x_{t+1}-x_{t}\|^{p+\nu-1},
\end{equation*}
for some $g_{\varphi}(x_{t+1})\in\partial\varphi(x_{t+1})$.\\
\noindent\textbf{Step 2.} Set $t:=t+1$ and go back to Step 1.
\end{mdframed}
For $p=3$, point $x_{t+1}$ at Step 1 can be computed by Algorithm 2 in \cite{GN4}, which is linearly convergent. As far as we know, the development of efficient algorithms to approximately solve (\ref{eq:aug5.3})-(\ref{eq:ale5.2}) with $p>3$ is still an open problem.
\begin{theorem}
\label{thm:ale5.1}
Suppose that H1 holds and that $\tilde{f}$ is bounded from below by $\tilde{f}^{*}$. Given $\epsilon\in (0,1)$, assume that $\left\{x_{t}\right\}_{t=0}^{T}$ is a sequence generated by Algorithm 3 such that $\|\nabla\tilde{f}(x_{t})\|_{*}>\epsilon$ for $t=0,\ldots,T$. Then,
\begin{equation}
T\leq 8(p+1)!M^{\frac{1}{p+\nu-1}}\left(\tilde{f}(x_{0})-\tilde{f}^{*}\right)\epsilon^{-\frac{p+\nu}{p+\nu-1}}.
\label{eq:ale5.12}
\end{equation}
\end{theorem}

\begin{proof}
By Lemma \ref{lem:ale5.2}, bound (\ref{eq:ale5.12}) follows as in Remark \ref{rem:ale1}.
\end{proof}

\subsection{Extended Accelerated Scheme}

Let us consider the following variant of Algorithm 2 for composite minimization:
\begin{mdframed}
\noindent\textbf{Algorithm 4. Two-Phase Accelerated Tensor Method}
\\[0.2cm]
\noindent\textbf{Step 0.} Choose $x_{0}\in\E$, $\theta\geq 0$ and $\epsilon\in (0,1)$. Define $\psi_{0}(x)=\frac{1}{p+\nu}\|x-x_{0}\|^{p+\nu}$. Set $v_{0}=z_{0}=x_{0}$, $A_{0}=0$, $M=p\left(H_{f,p}(\nu)+3\theta(p-1)!\right)$ and $t:=0$.\\
\noindent\textbf{Step 1.} If $t>0$ and $\min\left\{\|\nabla \tilde{f}(x_{t})\|_{*},\|\nabla \tilde{f}(z_{t})\|_{*}\right\}\leq\epsilon$, STOP.\\
\noindent\textbf{Step 2.} Compute the coefficient $a_{t}>0$ by solving the equation 
\begin{equation*}
a_{t}^{p+\nu}=\dfrac{1}{2^{(3p-1)}}\left[\dfrac{(p-1)!}{M}\right](A_{t}+a_{t})^{p+\nu-1}.
\end{equation*}
\noindent\textbf{Step 3.} Set $y_{t}=(1-\gamma_{t})x_{t}+\gamma_{t}v_{t}$, with $\gamma_{t}=a_{t}/[A_{t}+a_{t}]$.\\
\noindent\textbf{Step 4.} Compute an approximate solution $x_{t+1}$ to $\min_{x\in\E}\tilde{\Omega}_{y_{t},p,M}^{(\nu)}(x)$ such that
\begin{equation}
\tilde{\Omega}_{y_{t},p,M}^{(\nu)}(x_{t+1})\leq\tilde{f}(y_{t})\quad\text{and}\quad \|\nabla\Omega_{y_{t},p,M}^{(\nu)}(x_{t+1})+g_{\varphi}(x_{t+1})\|\leq\theta\|x_{t+1}-y_{t}\|^{p+\nu-1}.
\label{eq:ale5.13}
\end{equation}
for some $g_{\varphi}(x_{t+1})\in\partial\varphi(x_{t+1})$.\\
\noindent\textbf{Step 5.}  Define $\psi_{t+1}(x)=\psi_{t}(x)+a_{t}\left[f(x_{t+1})+\langle\nabla f(x_{t+1}),x-x_{t+1}\rangle + \varphi(x)\right]$ and compute $v_{t+1}=\arg\min_{x\in\E}\,\psi_{t+1}(x)$.\\
\noindent\textbf{Step 6.} Set $\bar{z}_{t}=\arg\min\left\{\tilde{f}(y)\,:\,y\in\left\{z_{t},x_{t+1}\right\}\right\}$ and $j:=0$.\\
\noindent\textbf{Step 7} Compute an approximate solution $z_{t+1}$ to $\min_{x\in\E}\tilde{\Omega}_{\bar{z}_{t},p,M}^{(\nu)}(x)$ such that
\begin{equation}
\tilde{\Omega}_{\bar{z}_{t},p,M}^{(\nu)}(z_{t+1})\leq\tilde{f}(\bar{z}_{t})\quad\text{and}\quad \|\nabla\Omega_{\bar{z}_{t},p,M}^{(\nu)}(z_{t+1})+g_{\varphi}(z_{t+1})\|\leq\theta\|z_{t+1}-\bar{z}_{t}\|^{p+\nu-1}.
\label{eq:ale5.14}
\end{equation}
for some $g_{\varphi}(z_{t+1})\in\partial\varphi(z_{t+1})$.\\
\noindent\textbf{Step 8} Set $t:=t+1$ and go to Step 1.
\end{mdframed}

The next theorem gives the global convergence rate for Algorithm 4 in terms of the norm of the gradient. Its proof is a direct adaptation of the proof of Theorem 4.2.

\begin{theorem}
\label{thm:ale5.4}
Suppose that H1 holds. Assume that $\left\{z_{t}\right\}_{t=0}^{T}$ is a sequence generated by Algorithm 4 such that
\begin{equation}
\|\nabla\tilde{f}(z_{t})\|_{*}>\epsilon,\quad t=0,\ldots,T.
\label{eq:ale5.15}
\end{equation}
If $T=2s$ for some $s>1$, then
\begin{equation}
\tilde{g}_{T}\equiv \min_{0\leq k\leq T}\|\nabla\tilde{f}(z_{t})\|_{*}\leq\dfrac{\left[2^{4(p+1)}\right]^{\frac{p+\nu-1}{p+\nu}}M\|x_{0}-x^{*}\|^{p+\nu-1}}{(T-2)^{\frac{(p+\nu-1)(p+\nu+1)}{p+\nu}}}.
\label{eq:ale5.16}
\end{equation}
Consequently,
\begin{equation}
T\leq 2+\left[2^{4(p+1)}\right]^{\frac{1}{p+\nu+1}}\|x_{0}-x^{*}\|^{\frac{p+\nu}{p+\nu+1}}\left(\dfrac{M}{\epsilon}\right)^{\frac{p+\nu}{(p+\nu-1)(p+\nu+1)}}.
\label{eq:ale5.17}
\end{equation}
\end{theorem} 

\begin{proof}
In view of Theorem \ref{thm:B2}, we have
\begin{equation}
\tilde{f}(x_{t})-\tilde{f}(x^{*})\leq\dfrac{2^{3p-1}M(p+\nu)^{p+\nu-1}\|x_{0}-x^{*}\|^{p+\nu}}{(p-1)!(t-1)^{p+\nu}},
\label{eq:ale5.18}
\end{equation}
for $t=2,\ldots,T$. On the other hand, by (\ref{eq:ale5.14}) and Lemma \ref{lem:ale5.2}, we have
\begin{equation}
\tilde{f}(\bar{z}_{t})-\tilde{f}(z_{t+1})\geq\dfrac{1}{8(p+1)!M^{\frac{1}{p+\nu-1}}}\|\nabla\tilde{f}(z_{t+1})\|_{*}^{\frac{p+\nu}{p+\nu-1}},
\label{eq:ale5.19}
\end{equation}
for $t=0,\ldots,T-1$. Thus, $f(z_{t+1})\leq f(\bar{z}_{t})\leq\min\left\{f(x_{t+1}),f(z_{t})\right\}$ and, consequently,
\begin{equation}
\tilde{f}(z_{t})\leq \tilde{f}(x_{t}),\quad t=0,\ldots,T,
\label{eq:ale5.20}
\end{equation}
and
\begin{equation}
\tilde{f}(z_{t})-\tilde{f}(z_{t+1})\geq\dfrac{1}{8(p+1)!M^{\frac{1}{p+\nu-1}}}\|\nabla\tilde{f}(z_{t+1})\|_{*}^{\frac{p+\nu}{p+\nu-1}},\quad t=0,\ldots,T-1.
\label{eq:ale5.21}
\end{equation}
Since $T=2s$, combining (\ref{eq:ale5.18}), (\ref{eq:ale5.20}) and (\ref{eq:ale5.21}), we obtain
\begin{eqnarray*}
\dfrac{2^{3p-1}M(p+\nu)^{p+\nu-1}\|x_{0}-x^{*}\|^{p+\nu}}{(p-1)!(s-1)^{p+\nu}}&\geq &\tilde{f}(x_{s})-\tilde{f}(x^{*})\geq \tilde{f}(z_{s})-\tilde{f}(x^{*})\\
& = & \tilde{f}(z_{T})-\tilde{f}(x^{*})+\sum_{k=s}^{T-1}\tilde{f}(z_{t})-\tilde{f}(z_{t+1})\\
&\geq & \dfrac{(s-1)}{8(p+1)!M^{\frac{1}{p+\nu-1}}}\left(\tilde{g}_{T}\right)^{\frac{p+\nu}{p+\nu-1}},
\end{eqnarray*} 
where $\tilde{g}_{T}=\min_{0\leq k\leq T}\|\nabla\tilde{f}(z_{t})\|_{*}$. Therefore,
\begin{equation*}
\left(\tilde{g}_{T}\right)^{\frac{p+\nu}{p+\nu-1}}\leq \dfrac{2^{4(p+1)}M^{\frac{p+\nu}{p+\nu-1}}(p+1)^{p+\nu+1}\|x_{0}-x^{*}\|^{p+\nu}}{(T-2)^{p+\nu+1}},
\end{equation*}
which gives (\ref{eq:ale5.16}). Finally, by (\ref{eq:ale5.15}) we have $\tilde{g}_{T}>\epsilon$. Thus, (\ref{eq:ale5.17}) follows directly from (\ref{eq:ale5.16}).
\end{proof}

\subsection{Regularization Approach}

Now, let us consider the ideal situation in which $\nu$, $H_{f,p}(\nu)$ and $R\geq \|x_{0}-x^{*}\|$ are known. In this case, a complexity bound with a better dependence on $\epsilon$ can be obtained by repeatedly applying an accelerated algorithm to a suitable regularization of $\tilde{f}$. Specifically, given $\delta>0$, consider the regularized problem
\begin{equation}
\min_{x\in\mathbb{R}^{n}}\,\tilde{F}_{\delta}(x)\equiv F_{\delta}(x)+\varphi(x),
\label{eq:5.17}
\end{equation}
for
\begin{equation}
F_{\delta}(x)=f(x)+\dfrac{\delta}{p+\nu}\|x-x_{0}\|^{p+\nu}.
\label{eq:july1}
\end{equation}

\begin{lemma}
Given $x_{0}\in\E$ and $\nu\in [0,1]$, let $d_{p+\nu}:\E\to\mathbb{R}$ be defined by $d_{p+\nu}(x)=\|x-x_{0}\|^{p+\nu}$, where $\|\,.\,\|$ is the Euclidean norm defined in (\ref{eq:0.1}). Then,
\begin{equation*}
\|D^{p}d_{p+\nu}(x)-D^{p}d_{p+\nu}(y)\|\leq C_{p,\nu}\|x-y\|^{\nu},\quad\forall x,y\in\E,
\end{equation*}
where $C_{p,\nu}=2\Pi_{i=1}^{p}(\nu+i)$.
\end{lemma}

\begin{proof}
See \cite{NES9}.
\end{proof}

As a consequence of the lemma above, we have the following property.
\begin{lemma}
\label{lem:june1}
If H1 holds, then the $p$th derivative of $F_{\delta}(\,.\,)$ in (\ref{eq:july1}) is $\nu$-H\"{o}lder continuous with constant $H_{F_{\delta},p}(\nu)=H_{f,p}(\nu)+\frac{\delta}{p+\nu}C_{p,\nu}$.

\end{lemma}

In view of Lemma \ref{lem:june1}, to solve (\ref{eq:5.17}) we can use the following instance of Algorithm A (see Appendix A):
\\[0.2cm]
\begin{mdframed}
\noindent\textbf{Algorithm 5. Accelerated Tensor Method for Problem (\ref{eq:5.17})}
\\[0.2cm]
\noindent\textbf{Step 0.} Choose $x_{0}\in\E$, $\theta\geq 0$ and $\epsilon\in (0,1)$. Define function \\$\psi_{0}(x)=\frac{1}{p+\nu}\|x-x_{0}\|^{p+\nu}$. Set 
\begin{equation}
H_{\delta}=p\left(H_{F_{\delta},p}+3\theta (p-1)!\right),
\label{eq:stela1}
\end{equation}
$v_{0}=x_{0}$, $A_{0}=0$ and $t:=0$.\\
\noindent\textbf{Step 1.} Compute the coefficient $a_{t}>0$ by solving equation 
\begin{equation*}
a_{t}^{p+\alpha}=\dfrac{1}{2^{(3p-1)}}\left[\dfrac{(p-1)!}{H_{\delta}}\right](A_{t}+a_{t})^{p+\alpha-1}.
\end{equation*}
\noindent\textbf{Step 2.} Compute $y_{t}=(1-\gamma_{t})x_{t}+\gamma_{t}v_{t}$, with $\gamma_{t}=a_{t}/[A_{t}+a_{t}]$.\\
\noindent\textbf{Step 3.} Compute an approximate solution $x_{t+1}$ to $\min_{x\in\E}\tilde{\Omega}_{y_{t},p,M}^{(\nu)}(x)$ such that
\begin{equation*}
\tilde{\Omega}_{y_{t},p,M}^{(\nu)}(x_{t+1})\leq\tilde{F}_{\delta}(y_{t})\quad\text{and}\quad \|\nabla\Omega_{y_{t},p,M}^{(\nu)}(x_{t+1})+g_{\varphi}(x_{t+1})\|\leq\theta\|x_{t+1}-y_{t}\|^{p+\nu-1}.
\end{equation*}
for some $g_{\varphi}(x_{t+1})\in\partial\varphi(x_{t+1})$.\\
\noindent\textbf{Step 4.} Define $\psi_{t+1}(x)=\psi_{t}(x)+a_{t}\left[F_{\delta}(x_{t+1})+\langle\nabla F_{\delta}(x_{t+1}),x-x_{t+1}\rangle + \varphi(x)\right]$ and compute $v_{t+1}=\arg\min_{x\in\E}\,\psi_{t+1}(x)$.\\
\noindent\textbf{Step 5.} Set $t:=t+1$ and go back to Step 1.
\end{mdframed}

Let us consider the following restart procedure based on Algorithm 5.
\\[0.2cm]
\begin{mdframed}
\noindent\textbf{Algorithm 6. Accelerated Regularized Tensor Method}
\\[0.2cm]
\noindent\textbf{Step 0.} Choose $x_{0}\in\E$, $\epsilon\in (0,1)$, $\theta\geq 0$ and $\delta>0$. Define  
\begin{equation}
m=1+\left\lceil\left(\dfrac{2^{4p+\nu-2}(p+\nu)^{p+\nu}H_{\delta}}{\delta (p-1)!}\right)^{\frac{1}{p+\nu}}\right\rceil,
\label{eq:5.18}
\end{equation}
for $H_{\delta}$ defined in (\ref{eq:stela1}). Set $y_{0}=x_{0}$, $u_{0}=x_{0}$ and $k:=0$.\\
\noindent\textbf{Step 1.} If $k>0$ and $\|\nabla \tilde{F}_{\delta}(u_{k})\|_{*}\leq \epsilon/2$, STOP.\\
\noindent\textbf{Step 2.} By applying Algorithm 5 to problem (\ref{eq:5.17}), with $x_{0}^{(k)}=y_{k}$, compute the first $m$ iterates $\left\{x_{t}^{(k)}\right\}_{t=0}^{m}$.\\
\noindent\textbf{Step 3.} Set $y_{k+1}=x_{m}^{(k)}$ and compute $u_{k+1}\in\mathbb{R}^{n}$ such that\small
\begin{equation}
\tilde{\Omega}_{y_{t},p,M}^{(\nu)}(u_{k+1})\leq\tilde{F}_{\delta}(y_{k+1})\quad\text{and}\quad \|\nabla\Omega_{y_{k+1},p,M}^{(\nu)}(u_{k+1})+g_{\varphi}(u_{k+1})\|\leq\theta\|u_{k+1}-y_{k+1}\|^{p+\nu-1}.
\label{eq:5.20}
\end{equation}
\normalsize
for some $g_{\varphi}(u_{k+1})\in\partial\varphi(u_{k+1})$.\\
\noindent\textbf{Step 4.} Set $k:=k+1$ and go back to Step 1.
\end{mdframed}

\begin{theorem}
\label{thm:7.1}
Suppose that H1 holds and let $\left\{u_{k}\right\}_{k=0}^{T}$ be a sequence generated by Algorithm 6 such that
\begin{equation}
\|\nabla \tilde{F}_{\delta}(u_{k})\|>\dfrac{\epsilon}{2},\quad k=0,\ldots,T.
\label{eq:5.22}
\end{equation}
Then,
\begin{equation}
T\leq 1+\log_{2}\left(\dfrac{32(p+1)!H_{\delta}^{\frac{1}{p+\nu-1}}\delta R^{p+\nu}}{2^{p+\nu-1}(p+\nu)\epsilon^{\frac{p+\nu}{p+\nu-1}}}\right).
\label{eq:5.23}
\end{equation}
\end{theorem}

\begin{proof}
Let $x_{\delta}^{*}=\arg\min_{x\in\E}\,\tilde{F}_{\delta}(x)$. By Theorem \ref{thm:B2} and (\ref{eq:5.18}), we have
\begin{eqnarray}
\tilde{F}_{\delta}(y_{k+1})-\tilde{F}_{\delta}(x_{\delta}^{*})& = & \tilde{F}_{\delta}(x_{m}^{(k)})-\tilde{F}_{\delta}(x_{\delta}^{*})\nonumber\\
 &\leq &\dfrac{2^{3p-1}H_{\delta}(p+\nu)^{p+\nu-1}\|x_{0}^{(k)}-x_{\delta}^{*}\|^{p+\nu}}{(p-1)!(m-1)^{p+\nu}}\nonumber\\
 &\leq &\dfrac{\delta 2^{-(p+\nu-2)}}{2(p+\nu)}\|y_{k}-x_{\delta}^{*}\|^{p+\nu}.
 \label{eq:5.24}
\end{eqnarray}
On the other hand, by Lemma 5 in \cite{DOI} and Lemma 1 in \cite{NES3}, function $F_{\delta}(\,.\,)$ is uniformly convex of degree $p+\nu$ with parameter $2^{-(p+\nu-2)}$. Thus,
\begin{equation}
\tilde{F}_{\delta}(y_{k+1})-\tilde{F}_{\delta}(x_{\delta}^{*})\geq\dfrac{\delta 2^{-(p+\nu-2)}}{p+\nu}\|y_{k+1}-x_{\delta}^{*}\|^{p+\nu}.
\label{eq:5.25}
\end{equation}
Combining (\ref{eq:5.24}) and (\ref{eq:5.25}), we obtain $\|y_{k+1}-x_{\delta}^{*}\|^{p+\nu}\leq\dfrac{1}{2}\|y_{k}-x_{\delta}^{*}\|^{p+\nu}$, and so
\begin{equation}
\|y_{k}-x_{\delta}^{*}\|^{p+\nu}\leq\left(\dfrac{1}{2}\right)^{k}\|y_{0}-x_{\delta}^{*}\|^{p+\nu}=\left(\dfrac{1}{2}\right)^{k}\|x_{0}-x_{\delta}^{*}\|^{p+\nu}.
\label{eq:5.26}
\end{equation}
Thus, it follows from (\ref{eq:5.24}) and (\ref{eq:5.26}) that
\begin{equation}
\tilde{F}_{\delta}(y_{k+1})-\tilde{F}_{\delta}(x_{\delta}^{*})\leq \dfrac{\delta}{2^{p+\nu-1}(p+\nu)}\left(\dfrac{1}{2}\right)^{k}\|x_{0}-x_{\delta}^{*}\|^{p+\nu}.
\label{eq:5.27}
\end{equation}
In view of Lemma \ref{lem:ale5.2}, by (\ref{eq:5.20}) and (\ref{eq:stela1}), we get
\begin{equation}
\tilde{F}_{\delta}(y_{k+1})-\tilde{F}_{\delta}(u_{k+1})\geq\dfrac{1}{8(p+1)!H_{\delta}^{\frac{1}{p+\nu-1}}}\|\nabla \tilde{F}_{\delta}(u_{k+1})\|_{*}^{\frac{p+\nu}{p+\nu-1}}.
\label{eq:5.28}
\end{equation}
Then, combining (\ref{eq:5.27}) and (\ref{eq:5.28}), it follows that
\begin{equation*}
\dfrac{1}{8(p+1)!H_{\delta}^{\frac{1}{p+\nu-1}}}\|\nabla \tilde{F}_{\delta}(u_{k+1})\|_{*}^{\frac{p+\nu}{p+\nu-1}}\leq\dfrac{\delta}{2^{p+\nu-1}(p+\nu)}\left(\dfrac{1}{2}\right)^{k}\|x_{0}-x_{\delta}^{*}\|^{p+\nu}.
\end{equation*}
In particular, for $k=T-1$, it follows from (\ref{eq:5.22}) that
\begin{equation}
2^{T-1}\leq\dfrac{32(p+1)!H_{\delta}^{\frac{1}{p+\nu-1}}\delta\|x_{0}-x_{\delta}^{*}\|^{p+\nu}}{2^{p+\nu-1}(p+\nu)\epsilon^{\frac{p+\nu}{p+\nu-1}}}.
\label{eq:stela2}
\end{equation}
Since $\tilde{F}_{\delta}(x_{\delta}^{*})\leq \tilde{F}_{\delta}(x^{*})$, it follows that $\|x_{0}-x_{\delta}^{*}\|\leq \|x_{0}-x^{*}\|\leq R$. Thus, combining this with (\ref{eq:stela2}), we get (\ref{eq:5.23}).
\end{proof}

\begin{corollary}
\label{cor:7.1}
Suppose that H1 holds and that $R\geq 1$. Then, Algorithm 6 with 
\begin{equation}
\delta=\dfrac{\epsilon}{2^{(p+\nu)}R^{p+\nu-1}}.
\label{eq:5.29}
\end{equation}
perform at most 
\begin{equation}
\mathcal{O}\left(\log_{2}\left(\dfrac{R^{p+\nu-1}}{\epsilon}\right)\left(\dfrac{R^{p+\nu-1}}{\epsilon}\right)^{\frac{1}{p+\nu}}\right).
\label{eq:5.30}
\end{equation}
iterations of Algorithm 5 in order to generate $u_{T}$ such that $\|\nabla \tilde{f}(u_{T})\|_{*}\leq\epsilon$.

\end{corollary}

\begin{proof}
By Theorem \ref{thm:7.1}, we can obtain $\|\nabla \tilde{F}_{\delta}(u_{T})\|_{*}\leq \epsilon/2$ with
\begin{equation}
T\leq 2+\log_{2}\left(\dfrac{32(p+1)!H_{\delta}^{\frac{1}{p+\nu-1}}\delta R^{p+\nu}}{2^{p+\nu-1}(p+\nu)\epsilon^{\frac{p+\nu}{p+\nu-1}}}\right).
\label{eq:5.31}
\end{equation}
Moreover, it follows from (\ref{eq:stela1}), (\ref{eq:5.29}), the definition of $H_{F_{\delta},p}(\nu)$ in Lemma \ref{lem:june1}, $\epsilon\in (0,1)$ and $R\geq 1$ that
\begin{eqnarray}
H_{\delta}&=& p\left(H_{F_{\delta},p}+3\theta (p-1)!\right)= p\left(H_{f,p}(\nu)+\dfrac{\delta}{p+\nu}C_{p,\nu}+3\theta (p-1)!\right)\nonumber\\
         &=& p\left(H_{f,p}(\nu)+\dfrac{\epsilon}{2^{(p+\nu)}R^{p+\nu-1}}\dfrac{C_{p,\nu}}{(p+\nu)}+3\theta (p-1)!\right)\nonumber\\
         &\leq & p\left(H_{f,p}(\nu)+C_{p,\nu}+3\theta (p-1)!\right).
         \label{eq:june5.32}
\end{eqnarray}
Combining (\ref{eq:5.31}), (\ref{eq:june5.32}) and (\ref{eq:5.29}), we have
\begin{equation}
T\leq 2+\log_{2}\left(\dfrac{32(p+1)!\left[p\left(H_{f,p}(\nu)+C_{p,\nu}+3\theta (p-1)!\right)\right]^{\frac{1}{p+\nu-1}}R}{2^{2(p+\nu-2)}(p+\nu)\epsilon^{\frac{1}{p+\nu-1}}}\right).
\label{eq:june5.33}
\end{equation}
At this point $u_{T}$, we have
\begin{equation}
\|\nabla \tilde{f}(u_{T})\|_{*} \leq \|\nabla \tilde{F}_{\delta}(u_{T})\|_{*}+\dfrac{\delta}{p+\nu}\|\nabla d_{p+\nu}(u_{T})\|\leq\dfrac{\epsilon}{2}+\delta \|u_{T}-x_{0}\|^{p+\nu-1}.
\label{eq:5.32}
\end{equation}
Since $\tilde{F}_{\delta}(\,.\,)$ is uniformly convex of degree $p+\nu$ with parameter $2^{-(p+\nu-2)}$, it follows from (\ref{eq:5.28}) and (\ref{eq:5.27}) that
\begin{eqnarray*}
\dfrac{\delta 2^{-(p+\nu-2)}}{p+\nu}\|u_{T}-x_{\delta}^{*}\|^{p+\nu}&\leq & \tilde{F}_{\delta}(u_{T})-\tilde{F}_{\delta}(x_{\delta}^{*})\leq  \tilde{F}_{\delta}(y_{T})-\tilde{F}_{\delta}(x_{\delta}^{*})\\
&\leq &\dfrac{\delta}{2^{p+\nu-1}(p+\nu)}\left(\dfrac{1}{2}\right)^{T-1}\|x_{0}-x_{\delta}^{*}\|^{p+\nu}.
\end{eqnarray*}
Therefore, $\|u_{T}-x_{\delta}^{*}\|\leq \|x_{0}-x_{\delta}^{*}\|$, and so
\begin{eqnarray}
\|u_{T}-x_{0}\|^{p+\nu-1}&\leq &\left[\|u_{T}-x_{\delta}^{*}\|+\|x_{\delta}^{*}-x_{0}\|\right]^{p+\nu-1}\nonumber\\
                         &\leq & 2^{p+\nu-1}\|x_{0}-x_{\delta}^{*}\|^{p+\nu-1}\label{eq:mais1}\\
                         & \leq & 2^{p+\nu-1}R^{p+\nu-1}.
\label{eq:5.33}
\end{eqnarray}
Now, combining (\ref{eq:5.32}), (\ref{eq:5.33}) and (\ref{eq:5.29}), we obtain
\begin{equation}
\|\nabla \tilde{f}(u_{T})\|\leq \dfrac{\epsilon}{2}+\dfrac{\epsilon}{2}=\epsilon.
\end{equation}
The conclusion is obtained by noticing that, for $\delta$ given in (\ref{eq:5.29}) we have 
\begin{eqnarray}
m&=&1+\left\lceil\left(\dfrac{2^{4p+\nu-2}(p+\nu)^{p+\nu}H_{\delta}}{\delta (p-1)!}\right)^{\frac{1}{p+\nu}}\right\rceil\nonumber\\
&\leq & 1+\left\lceil\left(\dfrac{2^{4p+\nu-2}(p+\nu)^{p+\nu}\left[p\left(H_{f,p}(\nu)+C_{p,\nu}+3\theta (p-1)!\right)\right]}{\delta (p-1)!}\right)^{\frac{1}{p+\nu}}\right\rceil\nonumber\\
& = & 1+\left\lceil\left(\dfrac{2^{5p+2\nu-2}(p+\nu)^{p+\nu}R^{p+\nu-1}\left[p\left(H_{f,p}(\nu)+C_{p,\nu}+3\theta (p-1)!\right)\right]}{\epsilon (p-1)!}\right)^{\frac{1}{p+\nu}}\right\rceil
\label{eq:june5.34}
\end{eqnarray}
Thus, (\ref{eq:5.30}) follows from multiplying (\ref{eq:june5.33}) and (\ref{eq:june5.34}).
\end{proof}

Suppose now that $S\geq \tilde{f}(x_{0})-\tilde{f}(x^{*})$ is known. In this case, we have the following variant of Theorem 5.7.
\begin{theorem}
Suppose that H1 holds and let $\left\{u_{k}\right\}_{k=0}^{T}$ be a sequence generated by Algorithm 6 such that
\begin{equation}
\|\nabla \tilde{F}_{\delta}(u_{k})\|_{*}>\dfrac{\epsilon}{2},\,\,k=0,\ldots,T.
\label{eq:mais5.35}
\end{equation}
Then,
\begin{equation}
T\leq 1+\log_{2}\left(\dfrac{16(p+1)!H_{\delta}^{\frac{1}{p+\nu-1}}S}{\epsilon^{\frac{p+\nu}{p+\nu-1}}}\right).
\label{eq:mais5.36}
\end{equation}
\end{theorem}

\begin{proof}
By (\ref{eq:stela2}), we have
\begin{equation}
T\leq 1+\log_{2}\left(\dfrac{32(p+1)!H_{\delta}^{\frac{1}{p+\nu-1}}}{2\epsilon^{\frac{p+\nu}{p+\nu-1}}}\dfrac{\delta}{2^{p+\nu-2}(p+\nu)}\|x_{0}-x_{\delta}^{*}\|^{p+\nu}\right).
\label{eq:mais5.37}
\end{equation}
Since $\tilde{F}_{\delta}(\,.\,)$ is uniformly convex of degree $p+\nu$ with parameter $\delta 2^{-(p+\nu-2)}$ we have
\begin{eqnarray}
\dfrac{\delta}{2^{p+\nu-2}(p+\nu)}\|x_{0}-x_{\delta}^{*}\|^{p+\nu}&\leq & \tilde{F}_{\delta}(x_{0})-\tilde{F}_{\delta}(x_{\delta}^{*})\nonumber\\
&=& \tilde{f}(x_{0})-\tilde{f}(x_{\delta}^{*})-\dfrac{\delta}{p+\nu}\|x_{\delta}^{*}-x_{0}\|^{p+\nu}\nonumber\\
&\leq & \tilde{f}(x_{0})-\tilde{f}(x_{\delta}^{*})\nonumber\\
&\leq & \tilde{f}(x_{0})-\tilde{f}(x^{*})\nonumber\\
&\leq & S.
\label{eq:mais5.38}
\end{eqnarray}
Combining (\ref{eq:mais5.37}) and (\ref{eq:mais5.38}) we get (\ref{eq:mais5.36}).
\end{proof}

\begin{corollary}
Suppose that H1 holds and that $S\geq 1$. Then, Algorithm 6 with 
\begin{equation}
\delta=\left[\dfrac{\epsilon}{2^{p+\nu}\left[2^{p+\nu-2}(p+\nu)S\right]^{\frac{p+\nu-1}{p+\nu}}}\right]^{p+\nu}
\label{eq:mais5.39}
\end{equation}
performs at most
\begin{equation}
\mathcal{O}\left(\log_{2}\left(\dfrac{S}{\epsilon^{\frac{p+\nu-1}{p+\nu}}}\right)\left(\dfrac{S^{\frac{p+\nu-1}{p+\nu}}}{\epsilon}\right)\right)
\label{eq:mais5.40}
\end{equation}
iterations of Algorithm 5 in order to generate $u_{T}$ such that $\|\nabla \tilde{f}(u_{T})\|_{*}\leq\epsilon$.
\end{corollary}

\begin{proof}
By Theorem 5.9, we can obtain $\|\nabla \tilde{F}_{\delta}(u_{T})\|_{*}\leq\epsilon/2$ with
\begin{equation}
T\leq 2+\log_{2}\left(\dfrac{16(p+1)!H_{\delta}^{\frac{1}{p+\nu-1}}S}{\epsilon^{\frac{p+\nu}{p+\nu-1}}}\right).
\label{eq:mais5.41}
\end{equation}
In view of (\ref{eq:mais5.39}), $\epsilon\in (0,1)$ and $S\geq 1$, we also have
\begin{equation}
H_{\delta}\leq p\left(H_{f,p}(\nu)+C_{p,\nu}+3\theta (p-1)!\right).
\label{eq:mais5.42}
\end{equation}
Thus, from (\ref{eq:mais5.41}) and (\ref{eq:mais5.42}) it follows that
\begin{equation}
T\leq 2+\log_{2}\left(\dfrac{16(p+1)!\left[p\left(H_{f,p}(\nu)+C_{p,\nu}+3\theta (p-1)!\right)\right]^{\frac{1}{p+\nu-1}}S}{\epsilon^{\frac{p+\nu}{p+\nu-1}}}\right).
\label{eq:mais5.43}
\end{equation}
At this point $u_{T}$ we have
\begin{equation}
\|\nabla \tilde{f}(u_{T})\|_{*}\leq \|\nabla \tilde{F}_{\delta}(u_{T})\|_{*}+\dfrac{\delta}{p+\nu}\|\nabla d_{p+\nu}(u_{T})\|_{*}\nonumber\\
\leq \dfrac{\epsilon}{2}+\delta\|u_{T}-x_{0}\|^{p+\nu-1}.
\label{eq:mais5.44}
\end{equation}
By (\ref{eq:mais1}) and (\ref{eq:mais5.38}),
\begin{eqnarray}
\|u_{T}-x_{0}\|^{p+\nu-1}&\leq & 2^{p+\nu-1}\|x_{0}-x_{\delta}^{*}\|^{p+\nu-1}\leq  2^{p+\nu-1}\left[\dfrac{2^{p+\nu-2}(p+\nu)S}{\delta}\right]^{\frac{p+\nu-1}{p+\nu}}\nonumber\\
                        & = &\left(\dfrac{1}{\delta}\right)^{\frac{p+\nu-1}{p+\nu}}2^{p+\nu-1}\left[2^{p+\nu-2}(p+\nu)S\right]^{\frac{p+\nu-1}{p+\nu}}.
\label{eq:mais5.45}
\end{eqnarray}
Thus, it follows from (\ref{eq:mais5.44}), (\ref{eq:mais5.45}) and (\ref{eq:mais5.39}) that
\begin{eqnarray*}
\|\nabla \tilde{f}(u_{T})\|_{*}&\leq &\dfrac{\epsilon}{2}+\delta^{\frac{1}{p+\nu}}2^{p+\nu-1}\left[2^{p+\nu-2}(p+\nu)S\right]^{\frac{p+\nu-1}{p+\nu}}\leq \dfrac{\epsilon}{2}+\dfrac{\epsilon}{2}=\epsilon.
\end{eqnarray*}
Finally, by (\ref{eq:5.18}) and (\ref{eq:mais5.39}) we have
\begin{eqnarray*}
m&=&1+\left\lceil\left(\dfrac{2^{4p+\nu-2}(p+\nu)^{p+\nu}H_{\delta}}{\delta (p-1)!}\right)^{\frac{1}{p+\nu}}\right\rceil\nonumber\\
&\leq &1+\left\lceil\left(\dfrac{2^{4p+\nu-2}(p+\nu)^{p+\nu}\left[pH_{f,p}(\nu)+C_{p,\nu}+3\theta (p-1)!\right]}{(p-1)!}\right)^{\frac{1}{p+\nu}}\dfrac{2^{p+\nu}\left[2^{p+\nu-2}(p+\nu)S\right]^{\frac{p+\nu-1}{p+\nu}}}{\epsilon}\right\rceil.
\end{eqnarray*}
Thus, (\ref{eq:mais5.40}) follows by multiplying (\ref{eq:mais5.43}) by the upper bound on $m$ given above.
\end{proof}

\section{Lower complexity bounds under H\"{o}lder condition}

In this section we derive lower complexity bounds for $p$-order tensor methods applied to the problem (\ref{eq:2.1}) in terms of the norm of the gradient of $f$, where the objective $f$ is convex and $H_{f,p}(\nu)<+\infty$ for some $\nu\in [0,1]$. 


For simplicity, assume that $\E=\mathbb{R}^{n}$ and $B=I_{n}$. Given an approximation $\bar{x}$ for the solution of (\ref{eq:2.1}), we consider $p$-order methods that compute trial points of the form $x^{+}=\bar{x}+\bar{h}$, where the search direction $\bar{h}$ is the solution of an auxiliary problem of the form
\begin{equation}
\min_{h\in\mathbb{R}^{n}}\,\phi_{a,\gamma,q}(h)\equiv \sum_{i=1}^{p}a^{(i)}D^{i}f(\bar{x})[h]^{i}+\gamma\|h\|^{q},
\label{eq:6.1}
\end{equation}
with $a\in\mathbb{R}^{p}$, $\gamma>0$ and $q>1$. Denote by $\Gamma_{\bar{x},f}(a,\gamma,q)$ the set of all stationary points of function $\phi_{a,\gamma,q}(\,.\,)$ and define the linear subspace
\begin{equation}
S_{f}(\bar{x})=\text{Lin}\left(\Gamma_{\bar{x},f}(a,\gamma,q)\,|\,a\in\mathbb{R}^{p},\,\gamma>0,\,q>1\right).
\label{eq:6.2}
\end{equation}
More specifically, we consider the class of $p$-order tensor methods characterized by the following assumption.
\\[0.2cm]
\noindent\textbf{Assumption 1.} Given $x_{0}\in\mathbb{R}^{n}$, the method generates a sequence of test points $\left\{x_{k}\right\}_{k\geq 0}$ such that
\begin{equation}
x_{k+1}\in x_{0}+\sum_{i=0}^{k}S_{f}(x_{i}),\quad k\geq 0.
\label{eq:6.3}
\end{equation}

Given $\nu\in [0,1]$, we consider the same family of difficult problems discussed in \cite{GN3}, namely:
\begin{equation}
f_{k}(x)=\dfrac{1}{p+\nu}\left[\sum_{i=1}^{k-1}|x^{(i)}-x^{(i+1)}|^{p+\nu}+\sum_{i=k}^{n}|x^{(i)}|^{p+\nu}\right]-x^{(1)},\,\,2\leq k\leq n.
\label{eq:6.4mais}
\end{equation}
The next lemma establishes that for each $f_{k}(\,.\,)$ we have $H_{f_{k},p}(\nu)<+\infty$.

\begin{lemma}
\label{lem:6.1}
Given an integer $k\in [2,n]$, the $p$th derivative of $f_{k}(\,.\,)$ is $\nu$-H\"{o}lder continuous with
\begin{equation}
H_{f_{k},p}(\nu)=2^{\frac{2+\nu}{2}}\Pi_{i=1}^{p-1}(p+\nu-i).
\label{eq:6.5}
\end{equation}
\end{lemma}

\begin{proof}
See Lemma 5.1 in \cite{GN3}.
\end{proof}

The next lemma provides additional properties of $f_{k}(\,.\,)$.

\begin{lemma}
\label{lem:6.2}
Given an integer $k\in [2,n]$, let function $f_{k}(\,.\,)$ be defined by (\ref{eq:6.4mais}). Then, $f_{k}(\,.\,)$ has a unique global minimizer $x_{k}^{*}$. Moreover, 
\begin{equation}
f_{k}^{*}=-\dfrac{(p+\nu-1)k}{p+\nu}\quad\text{and}\quad \|x_{k}^{*}\|<\dfrac{(k+1)^{\frac{3}{2}}}{\sqrt{3}}.
\label{eq:6.13}
\end{equation}
\end{lemma}

\begin{proof}
See Lemma 5.2 in \cite{GN3}.
\end{proof}

Our goal is to understand the behavior of the tensor methods specified by Assumption 1 when applied to the minimization of $f_{k}(\,.\,)$ with a suitable $k$. For that, let us consider the following subspaces:
\begin{equation*}
\mathbb{R}^{n}_{k}=\left\{x\in\mathbb{R}^{n}\,|\,x^{(i)}=0,\,\,i=k+1,\ldots,n\right\},\,\,1\leq k\leq n-1.
\end{equation*}

\begin{lemma}
\label{lem:6.3}
For any $q\geq 0$ and $x\in\mathbb{R}_{k}^{n}$, $f_{k+q}(x)=f_{k}(x)$.
\end{lemma}

\begin{proof}
It follows directly from (\ref{eq:6.4mais}).
\end{proof}

\begin{lemma}
\label{lem:6.4}
Let $\mathcal{M}$ be a $p$-order tensor method satisfying Assumption 1. If $\mathcal{M}$ is applied to the minimization of $f_{t}(\,.\,)$ ($2\leq t\leq n$) starting from $x_{0}=0$, then the sequence $\left\{x_{k}\right\}_{k\geq 0}$ of test points generated by $\mathcal{M}$ satisfies
\begin{equation*}
x_{k+1}\in\sum_{i=0}^{k}S_{f_{t}}(x_{i})\subset\mathbb{R}^{n}_{k+1},\,\,0\leq k \leq t-1.
\end{equation*}
\end{lemma}

\begin{proof}
See Lemma 2 in \cite{NES6}.
\end{proof}

The next lemma gives a lower bound for the norm of the gradient of $f_{t}(\,.\,)$ on suitable points.
\begin{lemma}
\label{lem:6.5}
Let $k$ be an integer in the interval $[1,t-1)$, with $t+1\leq n$. If $x\in\mathbb{R}^{n}_{k}$, then $\|\nabla f_{t}(x)\|_{*}\geq\frac{1}{\sqrt{k+1}}$.
\end{lemma}

\begin{proof}
In view of (\ref{eq:6.4mais}) we have
\begin{equation}
f_{k}(x)=\eta_{p+\nu}(A_{k}x)-\la e_{1},x\ra,
\label{eq:6.6}
\end{equation}
where
\begin{equation}
\eta_{p+\nu}(u)=\dfrac{1}{p+\nu}\sum_{i=1}^{n}|u^{(i)}|^{p+\nu},
\label{eq:6.7}
\end{equation}
and
\begin{equation}
A_{k}=\left(\begin{array}{cc} U_{k} & 0\\ 0 & I_{n-k} \end{array}\right),\quad\text{with}\quad U_{k}=\left(\begin{array}{rrrrrr} 1      & -1     &  0     & \ldots & 0      & 0\\
                                  0      &  1     & -1     & \ldots & 0      & 0\\
                                  \vdots & \vdots & \vdots &        & \vdots & \vdots\\
                                  0      &  0     &  0     & \ldots & 1      & -1\\
                                  0      &  0     &  0     & \ldots & 0      &  1 
             \end{array}\right)\in\mathbb{R}^{k\times k}.
\label{eq:6.8}
\end{equation}
By (\ref{eq:6.8}) and (\ref{eq:6.7}), we have
\begin{equation*}
\left(\nabla\eta_{p+\nu}(A_{t}x)\right)^{(i)}=\left\{\begin{array}{ll} |x^{(i)}-x^{(i+1)}|^{p+\nu-2}(x^{(i)}-x^{(i+1)}),& i=1,\ldots,t-1.\\
|x^{(i)}|^{p+\nu-2}(x^{(i)}),& i=t,\ldots,n.
\end{array}
\right.
\end{equation*}
Since $x\in\mathbb{R}^{n}_{k}$, it follows that $x^{(i)}=0$ for $i>k$. Therefore,
\begin{equation*}
\left(\nabla\eta_{p+\nu}(A_{t}x)\right)^{(i)}=0,\quad i=k+1,\ldots,n,
\end{equation*}
which means that $\nabla\eta_{p+\nu}(A_{t}x)\in\mathbb{R}^{n}_{k}$. Then, from (\ref{eq:6.6}), we obtain
\begin{eqnarray}
\|\nabla f_{t}(x)\|_{*}^{2}&=&\|A_{t}^{T}\nabla\eta_{p+\nu}(A_{t}x)-e_{1}\|_{*}^{2}\geq \inf_{y\in\mathbb{R}^{n}_{k}}\|A_{t}^{T}y-e_{1}\|_{*}^{2}\nonumber\\
                           &=& \inf_{z\in\mathbb{R}^{k}}\|Bz-e_{1}\|_{*}^{2} \quad\left[\text{where}\,\,B=A_{t}^{T}\left(\begin{array}{c}I_{k}\\ 0\end{array}\right)\right]\nonumber\\
                           &=& \|B(B^{T}B)^{-1}B^{T}e_{1}-e_{1}\|_{*}^{2}\nonumber\\
                           &=&\sum_{i=1}^{n}\left(\left[B(B^{T}B)^{-1}B^{T}e_{1}\right]^{(i)}-(e_{1})^{(i)}\right)^{2}.
\label{eq:6.17extra}
\end{eqnarray}
By (\ref{eq:6.8}), we have
\begin{equation}
B=\left(\begin{array}{c}\tilde{U}\\ 0\end{array}\right)\in\mathbb{R}^{n\times k},\,\,\text{with}\quad \tilde{U}=\left(\begin{array}{rrrrrr} 1 & 0 & 0 & \ldots & 0 & 0\\
                            -1 & 1 & 0 &\ldots & 0 & 0\\
                             0 & -1& 1& \ldots & 0 & 0\\
                             \vdots & \vdots & \vdots & & \vdots &\vdots\\
                             0  & 0 & 0 & \ldots & -1 & 1\\
                             0 & 0 & 0 & \ldots & 0 & 1
                             \end{array}\right)\in\mathbb{R}^{(k+1)\times k}
\label{eq:6.18extra}
\end{equation}
Consequently, 
\begin{equation}
B^{T}e_{1}=\left(\begin{array}{c} 1\\ 0\\ \vdots \\ 0\end{array}\right)\in\mathbb{R}^{k}.
\label{eq:6.19extra}
\end{equation}
and
\begin{equation}
B^{T}B=\left(\begin{array}{rrrrrrrr}2 & -1 & 0 & 0 &\ldots & 0 & 0 & 0\\
                                    -1& 2& -1& 0 &\ldots & 0 & 0 & 0\\
                                    \vdots & \vdots & \vdots & \vdots & & \vdots& \vdots & \vdots \\
                                    0 & 0 & 0 & 0 & \ldots & -1 & 2 & -1\\
                                    0 & 0 & 0 & 0 & \ldots & 0 & -1 & 2 \end{array}\right)\in\mathbb{R}^{k\times k}.
\label{eq:6.20extra}
\end{equation}
From (\ref{eq:6.20extra}), it can be checked that 
\begin{equation}
(B^{T}B)^{-1}=\dfrac{1}{k+1}\tilde{B}\in\mathbb{R}^{k\times k},
\label{eq:6.21extra}
\end{equation}
with 
\begin{equation}
\tilde{B}_{ij}=\left\{\begin{array}{ll} i[(k+1)-j],&\text{if}\,\,j\geq i,\\
                                        j[(k+1)-i],&\text{otherwise}.
                      \end{array}
                      \right.
\label{eq:6.22extra}
\end{equation}
Now, combining (\ref{eq:6.19extra}) and (\ref{eq:6.20extra})--(\ref{eq:6.21extra}), we get
\begin{equation}
\left[(B^{T}B)^{-1}B^{T}e_{1}\right]^{(i)}=\dfrac{(k+1)-i}{k+1},\quad i=1,\ldots,k.
\label{eq:6.23extra}
\end{equation}
Then, it follows from (\ref{eq:6.18extra}) and (\ref{eq:6.23extra}) that
\begin{eqnarray}
\left[B(B^{T}B)^{-1}B^{T}e_{1}\right]^{(i)}&=&\left\{\begin{array}{ll}\frac{k}{k+1},& i=1,\\
                                                  -\frac{(k+1)-(i-1)}{k+1}+\frac{(k+1)-i}{k+1},& i=2,\ldots,k,\\
                                                  -\frac{1}{k+1},& i=k+1,\\
                                                  0,& i=k+2,\ldots,n.
                                                  \end{array}
                                                  \right.\nonumber\\
                                           &=&\left\{\begin{array}{ll} 
                                           \frac{k}{k+1},& i=1,\\
                                           -\frac{1}{k+1},& i=2,\ldots,k+1,\\
                                           0,& i=k+2,\ldots,n.
                                           \end{array}
                                           \right.
\label{eq:6.24extra}
\end{eqnarray}
Finally, by (\ref{eq:6.17extra}) and (\ref{eq:6.24extra}) we have
\begin{eqnarray*}
\|\nabla f_{t}(x)\|_{*}^{2}&\geq &\sum_{i=1}^{n}\left(\left[B(B^{T}B)^{-1}B^{T}e_{1}\right]^{(i)}-(e_{1})^{(i)}\right)^{2}\\
&=& \left(-\dfrac{1}{k+1}\right)^{2}+\sum_{i=2}^{k+1}\left(-\dfrac{1}{k+1}\right)^{2}=\sum_{i=1}^{k+1}\dfrac{1}{(k+1)^{2}}\\
&=&\dfrac{1}{k+1},
\end{eqnarray*}
and the proof is complete.
\end{proof}

The next theorem establishes a lower bound for the rate of convergence of $p$-order tensor methods with respect to the  initial functional residual $(f(x_{0})-f^{*})$.

\begin{theorem}
\label{thm:extra6.2}
Let $\mathcal{M}$ be a $p$-order tensor method satisfying Assumption 1. Assume that for any function $f$ with $H_{f,p}(\nu)<+\infty$ this method ensures the rate of convergence:
\begin{equation}
\min_{1\leq k\leq t-1}\|\nabla f(x_{k})\|_{*}\leq\dfrac{H_{f,p}(\nu)^{\frac{1}{p+\nu}}(f(x_{0})-f^{*})^{\frac{p+\nu-1}{p+\nu}}}{\kappa(t)},\,\,t\geq 2,
\label{eq:extra6.21}
\end{equation}
where $\left\{x_{k}\right\}_{k\geq 0}$ is the sequence generated by method $\mathcal{M}$ and $f^{*}$ is the optimal value of $f$. Then, for all $t\geq 2$ such that $t+1\leq n$ we have
\begin{equation}
\kappa(t)\leq D_{p,\nu}t^{\frac{3(p+\nu)-2}{2(p+\nu)}}\quad\text{with}\quad D_{p,\nu}=\left[2^{\frac{2+\nu}{2}}\Pi_{i=1}^{p-1}(p+\nu-i)\right]^{\frac{1}{p+\nu}}\left[\dfrac{p+\nu-1}{p+\nu}\right]^{\frac{p+\nu-1}{p+\nu}}.
\label{eq:extra6.22}
\end{equation}
\end{theorem} 
\begin{proof}
Suppose that method $\mathcal{M}$ is applied to minimize function $f_{t}(\,.\,)$ with initial point $x_{0}=0$. By Lemma \ref{lem:6.4}, we have $x_{k}\in\mathbb{R}^{n}_{k}$ for all $k$, $1\leq k\leq t-1$. Thus, from Lemma \ref{lem:6.5} it follows that 
\begin{equation}
\min_{1\leq k\leq t-1}\|\nabla f_{t}(x_{k})\|_{*}\geq\min_{1\leq k\leq t-1}\dfrac{1}{\sqrt{k+1}}=\dfrac{1}{\sqrt{t}}.
\label{eq:extra6.24}
\end{equation}
Then, combining (\ref{eq:extra6.21}), (\ref{eq:extra6.24}), Lemma \ref{lem:6.1} and Lemma \ref{lem:6.2} we get
\begin{eqnarray*}
\kappa(t)&\leq &\dfrac{H_{f_{t},p}(\nu)^{\frac{1}{p+\nu}}(f_{t}(x_{0})-f_{t}^{*})}{\min_{1\leq k\leq t-1}\|\nabla f_{t}(x_{k})\|_{*}}\\
&\leq & \left[2^{\frac{2+\nu}{2}}\Pi_{i=1}^{p-1}(p+\nu-i)\right]^{\frac{1}{p+\nu}}\left[\dfrac{p+\nu-1}{p+\nu}\right]^{\frac{p+\nu-1}{p+\nu}}t^{\frac{p+\nu-1}{p+\nu}}t^{\frac{1}{2}}\\
&\leq & D_{p,\nu}(t+1)^{\frac{3(p+\nu)-2}{2(p+\nu)}},
\end{eqnarray*}
where constant $D_{p,\nu}$ is given in (\ref{eq:extra6.22}).
\end{proof}

\begin{remark}
Theorem \ref{thm:extra6.2} gives a lower bound of $\mathcal{O}\left(\left(\frac{1}{k}\right)^{\frac{3(p+\nu)-2}{2(p+\nu)}}\right)$ for the rate of convergence of tensor methods with respect to the initial functional residual. For first-order methods in the Lipschitz case (i.e., $p=\nu=1$), we have $\mathcal{O}\left(\frac{1}{k}\right)$. This gives a lower complexity bound of $\mathcal{O}(\epsilon^{-1})$ iterations for finding $\epsilon$-stationary points of convex functions using first-order methods, which coincides with the lower bound (8a) in \cite{CARMON}. Moreover, in view of Corollary 5.10, Algorithm 6 is suboptimal in terms of the initial residual, with the complexity a complexity gap that increases as $p$ grows. 
\end{remark}

Now, we obtain a lower bound for the rate of convergence of $p$-order tensor methods with respect to the distance $\|x_{0}-x^{*}\|$.
\begin{theorem}
\label{thm:6.2}
Let $\mathcal{M}$ be a $p$-order tensor method satisfying Assumption 1. Assume that for any function $f$ with $H_{f,p}(\nu)<+\infty$ this method ensures the rate of convergence:
\begin{equation}
\min_{1\leq k\leq t-1}\|\nabla f(x_{k})\|_{*}\leq\dfrac{H_{f,p}(\nu)\|x_{0}-x^{*}\|^{p+\nu-1}}{\kappa(t)},\,\,t\geq 2,
\label{eq:6.21}
\end{equation}
where $\left\{x_{k}\right\}_{k\geq 0}$ is the sequence generated by method $\mathcal{M}$ and $x^{*}$ is a global minimizer of $f$. Then, for all $t\geq 2$ such that $t+1\leq n$ we have
\begin{equation}
\kappa(t)\leq L_{p,\nu}(t+1)^{\frac{3(p+\nu)-2}{2}}\quad\text{with}\quad L_{p,\nu}=2^{\frac{2+\nu}{2}}(3)^{-\frac{p+\nu-1}{2}}\Pi_{i=0}^{p-1}(p+\nu-i).
\label{eq:6.22}
\end{equation}
\end{theorem} 

\begin{proof}
Let us apply method $\mathcal{M}$ for minimizing function $f_{t}(\,.\,)$ starting from point $x_{0}=0$. By Lemma \ref{lem:6.4}, we have $x_{k}\in\mathbb{R}^{n}_{k}$ for all $k$, $1\leq k\leq t-1$. Thus, from Lemma \ref{lem:6.5} it follows that 
\begin{equation}
\min_{1\leq k\leq t-1}\|\nabla f_{t}(x_{k})\|_{*}\geq\min_{1\leq k\leq t-1}\dfrac{1}{\sqrt{k+1}}=\dfrac{1}{\sqrt{t}}.
\label{eq:6.24}
\end{equation}
Then, combining (\ref{eq:6.21}), (\ref{eq:6.24}), Lemma \ref{lem:6.1} and Lemma \ref{lem:6.2} we get
\begin{eqnarray*}
\kappa(t)&\leq &\dfrac{H_{f_{t},p}(\nu)\|x_{0}-x_{t+1}^{*}\|^{p+\nu-1}}{\min_{1\leq k\leq t-1}\|\nabla f_{t}(x_{k})\|_{*}}\leq 2^{\frac{2+\nu}{2}}\Pi_{i=1}^{p-1}(p+\nu-i)\|x_{t}^{*}\|^{p+\nu-1}t^{\frac{1}{2}}\\
&\leq & 2^{\frac{2+\nu}{2}}\Pi_{i=1}^{p-1}(p+\nu-1)\left[\dfrac{(t+1)^{\frac{3}{2}}}{\sqrt{3}}\right]^{p+\nu-1}(t+1)^{\frac{1}{2}}\leq L_{p,\nu}(t+1)^{\frac{3(p+\nu)-2}{2}},
\end{eqnarray*}
where constant $L_{p,\nu}$ is given in (\ref{eq:6.22}).
\end{proof}

\begin{remark}
Theorem \ref{thm:6.2} establishes that the lower bound for the rate of convergence of tensor methods in terms of the norm of the gradient is also of $\mathcal{O}\left((\frac{1}{k})^{{3(p+\nu)-2\over 2}}\right)$. For first-order methods in the Lipschitz case (i.e., $p=\nu=1$) we have $\mathcal{O}\left(\frac{1}{k^{2}}\right)$. This gives a lower complexity bound of $\mathcal{O}(\epsilon^{-\frac{1}{2}})$ for finding $\epsilon$-stationary points of convex functions using first-order methods, which coincides with the lower bound (8b) in \cite{CARMON}.
\end{remark}

\begin{remark}
The rate of
$\mathcal{O}\left((\frac{1}{k})^{\frac{3(p+\nu)-2}{2}}\right)$
corresponds to a worst-case complexity bound of $\mathcal{O}\left(\epsilon^{-2/[3(p+\nu)-2]}\right)$
iterations necessary to ensure $\|\nabla f(x_{k})\|_{*}\leq\epsilon$. Note that, for $\epsilon\in (0,1)$, we have
\begin{equation*}
\left(\frac{1}{\epsilon}\right)^{\frac{p+\nu}{(p+\nu-1)(p+\nu+1)}}\leq\left(\dfrac{1}{\epsilon}\right)^{\frac{1}{p+\nu-1}}\leq \left(\frac{1}{\epsilon}\right)^{\frac{p+1}{(p-1)(3p-2)}}\left(\frac{1}{\epsilon}\right)^{\frac{2}{3(p+\nu)-2}}.
\end{equation*}
Thus, by increasing the power of the oracle (i.e., the order $p$), our non-universal schemes become nearly optimal. For example, if $\epsilon=10^{-6}$ and $p\geq 4$, we have 
$
\left(\frac{1}{\epsilon}\right)^{\frac{1}{p+\nu-1}}\leq 10\left(\frac{1}{\epsilon}\right)^{\frac{2}{3(p+\nu)-2}}.
$
\end{remark}
\section{Conclusion}

In this paper, we presented $p$-order methods that can find $\epsilon$-approximate stationary points of convex functions that are
$p$-times differentiable with $\nu$-H\"{o}lder continuous
$p$th derivatives. For the universal and the non-universal
schemes without acceleration, we established iteration
complexity bounds of
$\mathcal{O}\left(\epsilon^{-1/(p+\nu-1)}\right)$ for
finding $\bar{x}$ such that $\|\nabla f(\bar{x})\|_{*}\leq\epsilon$. For the case in which $\nu$ is known, we
obtain improved complexity bounds of of $\mathcal{O}\left(\epsilon^{-(p+\nu)/[(p+\nu-1)(p+\nu+1)]}\right)$ and $\mathcal{O}\left(|\log(\epsilon)|\epsilon^{-1/(p+\nu)}\right)$ for the corresponding
accelerated schemes. For the case in which $\nu$ is unknown, we obtained a bound of  $\mathcal{O}\left(\epsilon^{-(p+1)/[(p+\nu-1)(p+2)]}\right)$ for a universal accelerated scheme. Similar bounds were also obtained for tensor schemes adapted to the minimization of composite convex functions. A lower complexity bound of
$\mathcal{O}(\epsilon^{-2/[3(p+\nu)-2]})$ was obtained for
the referred problem class. Therefore, in practice, our non-universal schemes become nearly optimal as we increase the order $p$.

As an additional result, we showed that Algorithm 6 takes at most $\mathcal{O}\left(\log(\epsilon^{-1})\right)$ iterations to find $\epsilon$-stationary points of uniformly convex functions of degree $p+\nu$ in the form (5.64). Notice that strongly convex functions are uniformly convex of degree 2. Thus, our result generalizes the known bound of $\mathcal{O}\left(\log(\epsilon^{-1})\right)$ obtained for first-order schemes ($p=1$) applied to strongly convex functions with Lipschitz continuous gradients ($\nu=1$). At this point, it is not clear to us how $p$-order methods (with $p\geq 2$) behave when the objective functions is strongly convex with $\nu$-H\"{o}lder continuous $p$th derivatives. Neverthless, from the remarks done in \cite[p. 6]{DOI} for $p=2$, it appears that in our case the class of uniformly convex functions of degree $p+\nu$ is the most suitable for $p$-order methods from a physical point of view.

\section*{Acknowledgments}
The authors are very grateful to an anonymous referee, whose comments helped to improve the first version of this paper.

\section*{Funding}
G.N. Grapiglia was supported by the National Council for Scientific and Technological Development - Brazil (grant 406269/2016-5) and by the European Research Council Advanced Grant 788368. Yu. Nesterov was supported by the European Research Council Advanced Grant 788368.

\appendix

\section{Accelerated Scheme for Composite Minimization}

To solve problem (\ref{eq:ale5.1}), we can apply the following modification of Algorithm 3 in \cite{GN3}:
\begin{mdframed}
\noindent\textbf{Algorithm A. Accelerated Tensor Method for Composite Minimization}
\\[0.2cm]
\noindent\textbf{Step 0.} Choose $x_{0}\in\dom\varphi$, $\theta\geq 0$ and define $\psi_{0}(x)=\frac{1}{p+\nu}\|x-x_{0}\|^{p+\nu}$. Set $M\geq (p+\nu-1)\left(H_{f,p}(\nu)+\theta (p-1)!\right)$, $v_{0}=x_{0}$, $A_{0}=0$ and $t:=0$.\\
\noindent\textbf{Step 1.} Compute $a_{t}>0$ by solving the equation
\begin{equation}
a_{t}^{p+\nu}=\dfrac{1}{2^{(3p-1)}}\left[\dfrac{(p-1)!}{M}\right](A_{t}+a_{t})^{p+\nu-1}.
\label{eq:B1}
\end{equation}
\noindent\textbf{Step 2.} Compute $y_{t}=(1-\gamma_{t})x_{t}+\gamma_{t}v_{t}$ with $\gamma_{t}=a_{t}/[A_{t}+a_{t}]$.\\
\noindent\textbf{Step 3.} Compute an approximate solution $x_{t+1}$ to $\min_{x\in\E}\tilde{\Omega}_{y_{t},p,M}(x)$ such that
\begin{equation}
\tilde{\Omega}_{y_{t},p,M}(x_{t+1})\leq \tilde{f}(y_{t})\quad\text{and}\quad \|\nabla\Omega_{y_{t},p,M}(x_{t+1})+g_{\varphi}(x_{t+1})\|_{*}\leq\theta\|x_{t+1}-y_{t}\|^{p+\nu-1},
\label{eq:B2}
\end{equation}
for some $g_{\varphi}(x_{t+1})\in\partial\varphi (x_{t+1})$.\\
\noindent\textbf{Step 4.} Define $\psi_{t+1}(x)=\psi_{t}(x)+a_{t}\left[f(x_{t+1})+\la\nabla f(x_{t+1}),x-x_{t+1}\ra +\varphi(x)\right]$.\\
\noindent\textbf{Step 5.} Set $t:=t+1$ and go to Step 1.
\end{mdframed}

In order to establish a convergence rate for Algorithm B, we will need the following result.
\begin{lemma}
\label{lem:B1}
Suppose that H1 holds and let $x^{+}$ be an approximate solution to $\min_{y\in\E}\tilde{\Omega}_{x,p,H}^{(\nu)}(y)$ such that
\begin{equation}
\tilde{\Omega}_{x,p,H}^{(\nu)}(x^{+})\leq\tilde{f}(x)\quad\text{and}\quad \|\nabla\Omega_{x,p,H}^{(\nu)}(x^{+})+g_{\varphi}(x^{+})\|\leq\theta\|x^{+}-x\|^{p+\nu-1},
\label{eq:B5}
\end{equation}
for some $g_{\varphi}(x^{+})\in\partial\varphi(x^{+})$. If $H\geq (p+\nu-1)\left(H_{f,p}(\nu)+\theta (p-1)!\right)$, then
\begin{equation}
\la\nabla\tilde{f}(x^{+}),x-x^{+}\ra\geq\dfrac{1}{3}\left[\dfrac{(p-1)!}{H}\right]^{\frac{1}{p+\nu-1}}\|\nabla\tilde{f}(x^{+})\|_{*}^{\frac{p+\nu}{p+\nu-1}}.
\label{eq:B4}
\end{equation}
\end{lemma}

\begin{proof}
Denote $r=\|x^{+}-x\|$. Then,
\begin{eqnarray*}
\|\nabla\tilde{f}(x^{+})+\dfrac{H(p+\nu)}{p!}r^{p+\nu-2}B(x^{+}-x)\|_{*}& = & \|\nabla f(x^{+})-\nabla\Phi_{x,p}(x^{+})+\nabla\Omega_{x,p,H}^{(\nu)}(x^{+})+g_{\varphi}(x^{+})\|_{*}\\
&\leq &\|\nabla f(x^{+})-\nabla\Phi_{x,p}(x^{+})\|_{*}+\|\nabla\Omega_{x,p,H}^{(\nu)}(x^{+})+g_{\varphi}(x^{+})\|_{*}\\
&\leq &\left(\dfrac{H_{f,p}(\nu)}{(p-1)!}+\theta\right)r^{p+\nu-1},
\end{eqnarray*}
which gives
\begin{eqnarray}
\left(\dfrac{H_{f,p}(\nu)}{(p-1)!}+\theta\right)r^{2(p+\nu-1)}&\geq & \|\nabla\tilde{f}(x^{+})+\dfrac{H(p+\nu)}{p!}r^{p+\nu-2}B(x^{+}-x)\|_{*}^{2}\nonumber\\
& = & \|\nabla\tilde{f}(x^{+})\|_{*}^{2}+\dfrac{2(p+\nu)}{p!}Hr^{p+\nu-2}\la\nabla \tilde{f}(x^{+}),x^{+}-x\ra\nonumber\\ &  &+\dfrac{H^{2}(p+\nu)^{2}}{(p!)^{2}}r^{2(p+\nu-1)}.
\label{eq:B8}
\end{eqnarray}
From (\ref{eq:B8}), the rest of the proof follows exactly as in the proof of Lemma A.6 in \cite{GN3}.
\end{proof}

\begin{theorem}
\label{thm:B2}
Suppose that H1 holds and let the sequence $\left\{x_{t}\right\}_{t=0}^{T}$ be generated by Algorithm B. Then, for $t=2,\ldots,T$,
\begin{equation}
\tilde{f}(x_{t})-\tilde{f}(x^{*})\leq\dfrac{2^{3p-1}M(p+\nu)^{p+\nu}\|x_{0}-x^{*}\|^{p+\nu}}{(p-1)!(t-1)^{p+\nu}}.
\label{eq:B9}
\end{equation}
\end{theorem}

\begin{proof}
For all $t\geq 0$, we have
\begin{equation}
\psi_{t}(x)\leq A_{t}\tilde{f}(x)+\dfrac{1}{p+\nu}\|x-x_{0}\|^{p+\nu},\quad\forall x\in\E.
\label{eq:B10}
\end{equation}
Indeed, (\ref{eq:B10}) is true for $t=0$ because $A_{0}=0$ and $\psi_{0}(x)=\frac{1}{p+\nu}\|x-x_{0}\|^{p+\nu}$. Suppose that (\ref{eq:B10}) is true for some $t\geq 0$. Then,
\begin{eqnarray*}
\psi_{t+1}(x)& = & \psi_{t}(x)+a_{t}\left[f(x_{t+1})+\la\nabla f(x_{t+1}),x-x_{t+1}\ra+\varphi(x)\right]\\
 &\leq & A_{t}\tilde{f}(x)+a_{t}\tilde{f}(x)+\frac{1}{p+\nu}\|x-x_{0}\|^{p+\nu} = A_{t+1}\tilde{f}(x)+\dfrac{1}{p+\nu}\|x-x_{0}\|^{p+\nu}.
\end{eqnarray*}
Thus, (\ref{eq:B10}) follows by induction. Now, let us prove that
\begin{equation}
A_{t}\tilde{f}(x_{t})\leq\psi_{t}^{*}\equiv\min_{x\in\E}\psi_{t}(x).
\label{eq:B11}
\end{equation}
Again, using $A_{0}=0$, we see that (\ref{eq:B11}) is true for $t=0$. Assume that (\ref{eq:B11}) is true for some $t\geq 0$. Note that $\psi_{t}(\,.\,)$ is uniformly convex of degree $p+\nu$ with parameter $2^{-(p+\nu-2)}$. Thus, by the induction assumption 
\begin{equation*}
\psi_{t}(x)\geq \psi_{t}^{*}+\dfrac{2^{-(p+\nu-2)}}{p+\nu}\|x-v_{t}\|^{p+\nu}\\
           \geq  A_{t}\tilde{f}(x_{t})+\dfrac{2^{-(p+\nu-2)}}{p+\nu}\|x-v_{t}\|^{p+\nu}.
\end{equation*}
Consequently,
\begin{eqnarray}
\psi_{t+1}^{*}& = & \min_{x\in\dom\varphi}\left\{\psi_{t}(x)+a_{t}\left[f(x_{t+1})+\la\nabla f(x_{t+1}),x-x_{t+1}+\varphi(x)\right]\right\}\nonumber\\
&\geq &\min_{x\in\dom\varphi}\left\{A_{t}\tilde{f}(x_{t})+\dfrac{2^{-(p+\nu-2)}}{p+\nu}\|x-v_{t}\|^{p+\nu}\right.\nonumber\\
& & \left. +a_{t}\left[f(x_{t+1})+\la\nabla f(x_{t+1}),x-x_{t+1}\ra+\varphi(x)\right]\right\}.
\label{eq:B12}
\end{eqnarray}
Since $f$ is convex and differentiable and $g_{\varphi}(x_{t+1})\in\partial\varphi(x_{t+1})$, we have
\begin{equation}
\tilde{f}(x_{t})\geq \tilde{f}(x_{t+1})+\la\nabla\tilde{f}(x_{t+1}),x_{t}-x_{t+1}\ra
\label{eq:B13}
\end{equation}
and
\begin{equation}
\varphi(x)\geq\varphi(x_{t+1})+\la g_{\varphi}(x_{t+1}),x-x_{t+1}\ra.
\label{eq:B14}
\end{equation}
Using (\ref{eq:B13}) and (\ref{eq:B14}) in (\ref{eq:B12}), it follows that
\begin{eqnarray}
\psi_{t+1}^{*}&\geq &\min_{x\in\dom\varphi}\left\{A_{t+1}\tilde{f}(x_{t+1})+\la\nabla\tilde{f}(x_{t+1}),A_{t}x_{t}-A_{t}x_{t+1}\ra\right.\nonumber\\
& & \left. +a_{t}\la\nabla\tilde{f}(x_{t+1}),x-x_{t+1}\ra+\dfrac{2^{-(p+\nu-2)}}{p+\nu}\|x-v_{t}\|^{p+\nu}\right\}.
\label{eq:B15}
\end{eqnarray}
Note that $A_{t}x_{t}=A_{t+1}y_{t}-a_{t}v_{t}$ and $A_{t+1}x_{t+1}=A_{t}x_{t+1}+a_{t}x_{t+1}$. Thus, combining (\ref{eq:B15}) and Lemma \ref{lem:B1}, we obtain
\begin{eqnarray*}
\psi_{t+1}^{*}&\geq &
A_{t+1}\tilde{f}(x_{t+1})+\min_{x\in\dom\varphi}\left\{A_{t+1}\frac{1}{4}\left[\dfrac{(p-1)!}{M}\right]^{\frac{1}{p+\nu-1}}\|\nabla\tilde{f}(x_{t+1})\|_{*}^{\frac{p+\nu}{p+\nu-1}}\right.\\
& &\left. +a_{t}\la\nabla\tilde{f}(x_{t+1}),x-v_{t}\ra+\dfrac{2^{-(p+\nu-2)}}{p+\nu}\|x_{t}-v_{t}\|^{p+\nu}\right\}\geq A_{t+1}\tilde{f}(x_{t+1}),
\end{eqnarray*}
where the last inequality follows from (\ref{eq:B1}) exactly as in the proof of Theorem 4.2 in \cite{GN}. Thus, (\ref{eq:B11}) also holds for $t+1$, which completes the induction argument. 

Now, combining (\ref{eq:B10}) and (\ref{eq:B11}) we have
\begin{equation}
\tilde{f}(x_{t})-\tilde{f}(x^{*})\leq\dfrac{1}{A_{t}}\left[\dfrac{1}{p+\nu}\|x_{0}-x^{+}\|^{p+\nu}\right].
\label{eq:B18}
\end{equation}
Once again, as in the proof of Theorem 4.2 in \cite{GN3}, it follows from (\ref{eq:B1}) that
\begin{equation}
A_{t}\geq\dfrac{(p-1)!}{2^{3p-1}M}\left[\dfrac{1}{p+\nu}\left(\dfrac{1}{2}\right)^{\frac{p+\nu-1}{p+\nu}}\right]^{p+\nu}(t-1)^{p+\nu},\quad\forall t\geq 2.
\label{eq:B19}
\end{equation}
Finally, (\ref{eq:B9}) follows directly from (\ref{eq:B18}) and (\ref{eq:B19}).
\end{proof}

\end{document}